\documentclass[a4paper,10pt]{amsart}
\usepackage[latin1]{inputenc}
\usepackage{amssymb,colordvi,graphicx,tikz,vmargin,hyperref,overpic,graphics,ifthen}
\usetikzlibrary{patterns,calc}

\setmargrb{1in}{1in}{1in}{1in}
\newtheorem{theorem}{Theorem}[section]
\newtheorem{proposition}[theorem]{Proposition}
\newtheorem{lemma}[theorem]{Lemma}
\newtheorem{corollary}[theorem]{Corollary}
\newtheorem{definition}[theorem]{Definition}

\theoremstyle{definition}
\newtheorem{example}[theorem]{Example}
\newtheorem{remark}[theorem]{Remark}

\newcommand{\Z}{\mathbb{Z}}

\newcommand{\disk}{D} 
\newcommand{\configD}{\mathbf D} 

\newcommand{\diagD}[2]{\ifthenelse{\equal{#2}{L}}{#1^{\textsc l}}{#1^{\textsc r}}} 
\newcommand{\pseudoquadrangle}{\raisebox{-1pt}{\includegraphics{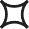}}} 
\newcommand{\ssm}{\smallsetminus} 

\DeclareMathOperator{\conv}{conv} 
\DeclareMathOperator{\val}{val}
\DeclareMathOperator{\Val}{Val}
\DeclareMathOperator{\ini}{in}
\DeclareMathOperator{\Web}{Web}
\DeclareMathOperator{\trop}{trop}
\DeclareMathOperator{\Trop}{Trop}
\DeclareMathOperator{\Gr}{Gr}
\DeclareMathOperator{\Dr}{Dr}

\newcommand{\ie}{\textit{i.e.}~} 
\newcommand{\eg}{\textit{e.g.}~} 

\definecolor{darkblue}{rgb}{0,0,0.7} 
\newcommand{\darkblue}{\color{darkblue}} 
\newcommand{\Dfn}[1]{\emph{\darkblue #1}} 

\newcommand\T{\rule{0pt}{2ex}}       
\newcommand\B{\rule[-2ex]{0pt}{0pt}} 

\graphicspath{{figures/}} 

\title{Cluster Algebras of Type $D_4$, Tropical Planes, and the Positive Tropical Grassmannian}

\author[S.B.~Brodsky]{Sarah~B.~Brodsky$^{\otimes}$}
\address[S.B.~Brodsky]{Department of Mathematics, Technische Universität Berlin, 10623 Berlin, Deutschland}
\email{\href{mailto:brodsky@math.tu-berlin.de}{\texttt{brodsky@math.tu-berlin.de}}}
\thanks{$^{\otimes}$With the support of the European Research Council grant SHPEF awarded to
Olga Holtz}

\author[C.~Ceballos]{Cesar Ceballos$^{\oplus}$}
\address[C.~Ceballos]{Faculty of Mathematics, University of Vienna, 1090 Vienna, Austria}
\email{\href{mailto:cesar.ceballos@univie.ac.at}{\texttt{cesar.ceballos@univie.ac.at}}}
\urladdr{\url{http://garsia.math.yorku.ca/~ceballos/}}
\thanks{$^{\oplus}$With partial support of the government of Canada through a
Banting Postdoctoral Fellowship, of a York University research grant, and of
the Austrian Science Foundation FWF, grant F 5008-N15, in the framework of the
Special Research Program ``Algorithmic and Enumerative Combinatorics''.}

\author[J.-P.~Labb\'e]{Jean-Philippe Labb\'e$^{\odot}$}
\address[J.-P. Labb\'e]{Einstein Institute of Mathematics, Hebrew University of Jerusalem, 91904 Jerusalem, Israel}
\email{\href{mailto:labbe@math.fu-berlin.de}{\texttt{labbe@math.fu-berlin.de}}}
\urladdr{\url{http://www.math.huji.ac.il/~labbe}}
\thanks{$^{\odot}$With the support of a FQRNT post-doctoral fellowship and a post-doctoral ISF grant (805/11)}


\begin{document}

\maketitle

\begin{abstract}
We show that the number of combinatorial types of clusters of type $D_4$ modulo reflection-rotation is exactly equal to the number of combinatorial types of tropical planes in $\mathbb{TP}^5$. This follows from a result of Sturmfels and Speyer which classifies these tropical planes into seven combinatorial classes using a detailed study of the tropical Grassmannian $\Gr(3,6)$.
Speyer and Williams show that the positive part $\Gr^+(3,6)$ of this tropical Grassmannian is combinatorially equivalent to a small coarsening of the cluster fan of type~$D_4$. We provide a structural bijection between the rays of $\Gr^+(3,6)$ and the almost positive roots of type $D_4$ which makes this connection more precise.  
This bijection allows us to use the pseudotriangulations model of the cluster algebra of type $D_4$ to 
describe the equivalence of ``positive" tropical planes in $\mathbb{TP}^5$, giving a combinatorial model which characterizes the combinatorial types of tropical planes using automorphisms of pseudotriangulations of the octogon.
\end{abstract}

\section{Introduction}
Cluster algebras are commutative rings generated by a set of cluster variables, which are grouped into overlapping sets called clusters. They were introduced by S.~Fomin and A.~Zelevinsky in the series of papers~\cite{FominZelevinsky-ClusterAlgebrasI, FominZelevinsky-ClusterAlgebrasII, FominZelevinsky-ClusterAlgebrasIII, FominZelevinsky-ClusterAlgebrasIV}, and since then have shown fascinating connections with diverse areas such as Lie theory, representation theory, Poisson geometry, algebraic geometry, combinatorics and discrete geometry. One important family of examples is the family of cluster algebras of finite type, which were classified in~\cite{FominZelevinsky-ClusterAlgebrasII} using the Cartan--Killing classification for finite crystallographic root systems. Among them are the cluster algebras of type $D_n$ which are one of the main objects of study in this paper.
We focus on the cluster algebra of type $D_4$ and the combinatorial structures related to its cluster complex. Cluster complexes are simplicial complexes that encode the mutation graph of the cluster algebra, i.e., the graph describing how to pass from a cluster to another. They were introduced by Fomin and Zelevinsky~\cite{FominZelevinsky-YSystems} in connection with their proof of the Zamolodchikov's periodicity conjecture for algebraic $Y$-systems. One remarkable connection of cluster complexes with tropical geometry was discovered by Speyer and Williams in~\cite{speyer_tropical_2005}.
They study the positive part~$\Gr^+(d,n)$ of the tropical Grassmannian $\Gr(d,n)$, and show that $\Gr^+(2,n)$ is combinatorially the cluster complex of type $A_{n-3}$, and that $\Gr^+(3,6)$ and $\Gr^+(3,7)$ are closely related to the cluster complexes of type $D_4$ and $E_6$.
The tropical Grassmannian $\Gr(d,n)$ was first introduced by Speyer and Sturmfels~\cite{speyer_tropical_2004} as a parametrization space for tropicalizations of ordinary linear spaces. 
The tropicalization of an ordinary linear space gives a tropical linear space in $\mathbb {TP}^{n-1}$ in the sense of Speyer~\cite{speyer_tropical_2008,speyer_matroid_2009}, but not all tropical linear spaces are realized in this way in general. In the case $\Gr(3,6)$, all tropical planes in $\mathbb{TP}^5$ are realized by the Grassmannian~\cite{speyer_tropical_2004}. Speyer and Sturmfels~\cite{speyer_tropical_2004} explicitly studied all tropical planes in $\mathbb{TP}^5$ and found that there are exactly 7 different combinatorial types. 
On the other hand, one may also classify clusters of type $D_4$ up to combinatorial type, from which we deduce that there are exactly 7 different combinatorial types of clusters modulo reflection and rotation. 
This leads to two natural questions:
\begin{enumerate}
\item How are the 7 combinatorial types of tropical planes in $\mathbb{TP}^5$ and the 7 combinatorial types of type $D_4$ clusters related?
\item Which of the 7 combinatorial types of tropical planes in $\mathbb{TP}^5$ are realized in the positive part $\Gr^+(3,6)$ of the tropical Grassmannian $\Gr(3,6)$?
\end{enumerate}

This paper gives precise answers to these questions.
Surprisingly, the 7 combinatorial types of tropical planes and the 7 combinatorial types of clusters are not bijectively related as one might expect.  We show that only 6 of the 7 combinatorial types of tropical planes are achieved by the positive tropical Grassmannian $\Gr^+(3,6)$, and use the pseudotriangulation model of cluster algebras of type $D_4$ to compare them with the combinatorial types of clusters of type $D_4$. 
In particular, we obtain that 

\begin{quote}
\textit{if two pseudotriangulations are related by a sequence of reflections of the octagon preserving the parity of the vertices, and possibly followed by a global exchange of central chords, then their corresponding tropical planes in $\mathbb{TP}^5$ are combinatorially equivalent.
}
\end{quote}

The combinatorial classes of positive tropical planes are then obtained by taking unions of the classes generated by this finer equivalence on pseudotriangulations.

The outline of the paper is as follows.
Section~\ref{sec:clusteralgebras} recalls the notion of cluster algebras of type $D_n$ and their cluster complexes, and shows that the clusters of type $D_4$ are divided into 7 combinatorial classes.
Section~\ref{sec:tropical} recalls the necessary notions in tropical geometry, tropical Grassmannians and their positive parts, as well as the Speyer--Sturmfels classification of tropical planes in $\mathbb{TP} ^5$ into 7 combinatorial classes.
In Section~\ref{sec:connecting}, we present a precise connection between the cluster complex of type~$D_4$ and the positive part $\Gr^+(3,6)$ of the tropical Grassmannian $\Gr(3,6)$.
We compute the combinatorial types of tropical planes in $\mathbb{TP}^5$ corresponding to the clusters of type $D_4$ in Section~\ref{sec:computations}.
Finally, in Section~\ref{sec:comparing} we describe the combinatorial types of tropical planes using automorphisms of pseudotriangulations.

{\bf Acknowledgements.} We are grateful to York University for hosting visits
of the first and third authors. We also thank Hugh Thomas for helpful
discussions.

\section{Cluster Algebras of type $D_n$}
\label{sec:clusteralgebras}

In this section, we present the cluster algebras of type~$D_n$. Different models exist for these cluster algebras: in terms of centrally symmetric triangulations of a $2n$-gon with bicolored long diagonals~\cite[Section~3.5]{FominZelevinsky-YSystems}\cite[Section~12.4]{FominZelevinsky-ClusterAlgebrasII}, in terms of tagged triangulations of a punctured $n$-gon~\cite{FominShapiroThurston}, or in terms of pseudotriangulations of a $2n$-gon with a small disk in the center~\cite{CP_pseudotriangulations}. 
We adopt the last model mentioned to deal with clusters combinatorially.
The automorphisms of pseudotriangulations allow us to define the combinatorial type of a cluster.
In this paper, we classify and compare clusters up these combinatorial types.
We refer to the original paper~\cite{CP_pseudotriangulations} for a more detailed study of cluster algebras of type~$D_n$ in terms of pseudotriangulations, and briefly describe here their model. 
 
Consider a regular convex~$2n$-gon together with a disk~$\disk$ placed at the center, whose radius is small enough such that~$\disk$ only intersects the long diagonals of the $2n$-gon. We denote by~$\configD_n$ the resulting configuration. 
The vertices of~$\configD_n$ are labeled by $0,1,\dots, n-1, \overline 0, \overline 1,\dots ,\overline{n-1}$ in counterclockwise direction, such that two vertices $p$ and $\overline p$ are symmetric with respect to the center of the polygon. 
The \Dfn{chords} of~$\configD_n$ are all the diagonals of the $2n$-gon, except the long ones, and all the segments tangent to the disk~$\disk$ and with one endpoint among the vertices of the $2n$-gon; see Figure~\ref{fig:configurationDn&Flips}.

\begin{figure}[!tbhp]
\includegraphics[scale=0.8]{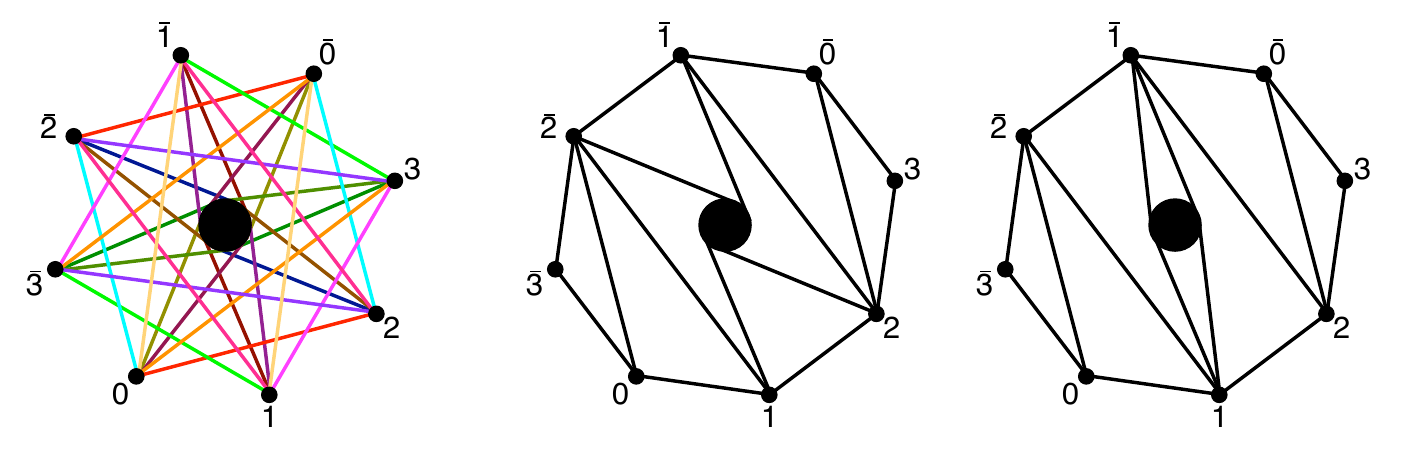}
\caption{The configuration~$\configD_4$ has $16$ centrally symmetric pairs of chords (left). A centrally symmetric pseudotriangulation~$T$ of~$\configD_4$ (middle). The centrally symmetric pseudotriangulation of~$\configD_4$ obtained from~$T$ by flipping the centrally symmetric pair of chords~$\{2^\textsc{l}, \overline 2^\textsc{l}\}$~(right).}
\label{fig:configurationDn&Flips}
\end{figure}

Each vertex~$p$ is incident to two such chords; we denote by~$\diagD{p}{L}$ (resp.~by~$\diagD{p}{R}$) the chord emanating from~$p$ and tangent on the left (resp.~right) to the disk~$\disk$. We call these chords \Dfn{central}.

Cluster variables, clusters, cluster mutations and exchange relations in the cluster algebra of type~$D_n$ can be interpreted geometrically as follows:
\begin{enumerate}
\item Cluster variables correspond to \Dfn{centrally symmetric pairs of (internal) chords} of the geometric configuration~$\configD_n$, as shown in Figure~\ref{fig:configurationDn&Flips}\,(left). To simplify notations, we identify a chord~$\delta$, its centrally symmetric copy~$\bar \delta$, and the pair~$\{\delta, \bar \delta\}$. 

\item Clusters correspond to \Dfn{centrally symmetric pseudotriangulations} of~$\configD_n$ (\ie maximal centrally symmetric crossing-free sets of chords of~$\configD_n$). Each pseudotriangulation of~$\configD_n$ contains exactly~$2n$ chords, and partitions $\conv(\configD_n) \ssm \disk$ into \Dfn{pseudotriangles}  (\ie interiors of simple closed curves with three convex corners related by three concave chains); see Figure~\ref{fig:configurationDn&Flips}. 

\item Cluster mutations correspond to \Dfn{flips} of centrally symmetric pairs of chords between centrally symmetric pseudotriangulations of~$\configD_n$. A flip in a pseudotriangulation~$T$ replaces an internal chord~$e$ by the unique other internal chord~$f$ such that~$(T \ssm e) \cup f$ is again a pseudotriangulation of~$T$. More precisely, deleting~$e$ in~$T$ merges the two pseudotriangles of~$T$ incident to~$e$ into a pseudoquadrangle~$\pseudoquadrangle$ (\ie the interior of a simple closed curve with four convex corners related by four concave chains), and adding~$f$ splits the pseudoquadrangle~$\pseudoquadrangle$ into two new pseudotriangles. The chords~$e$ and~$f$ are the two unique chords which lie both in the interior of~$\pseudoquadrangle$ and on a geodesic between two opposite corners of~$\pseudoquadrangle$. 

\item As in type~$A$, the exchange relations between cluster variables during a cluster mutation can be understood in the geometric picture. This is illustrated for all different kinds of flips in Figure~\ref{fig:typeDflip}.
\end{enumerate}

\begin{figure}[!htbp]
	\centerline{\includegraphics[width=0.85\textwidth]{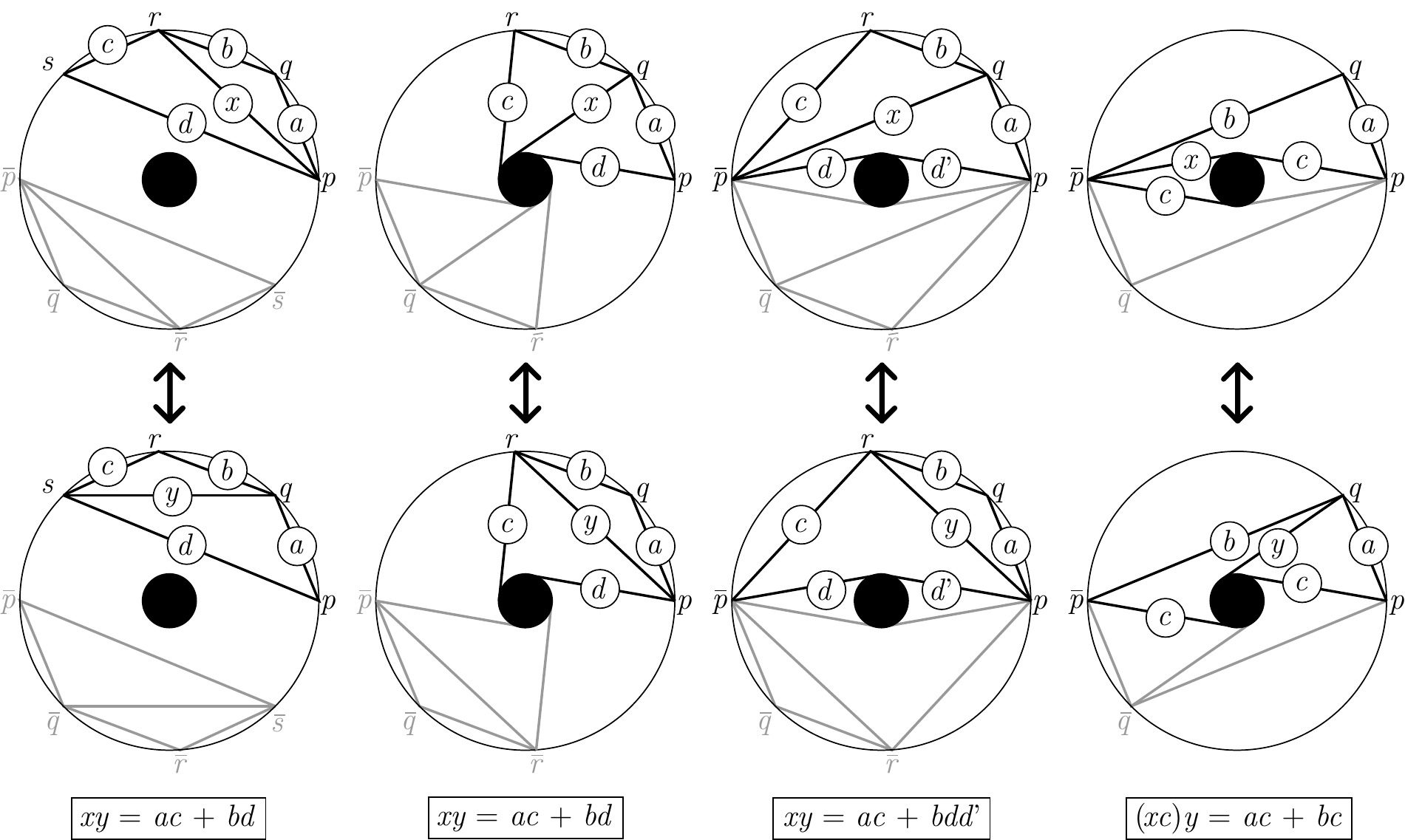}}
	\caption{Different kinds of flips and exchange relations in type~$D_n$}
	\label{fig:typeDflip}
\end{figure}

\subsection{The cluster complex of type $D_4$}
The cluster complex of type $D_n$ can also be described from the geometric model presented in~\cite{CP_pseudotriangulations}. We recall this description in the particular case of cluster algebras of type~$D_4$. 
Let $\{\tau_1,\tau_2,\tau_3,\tau_4\}$ be the set of simple generators of the Coxeter group of type $D_4$ according to the labeling of the Dynkin diagram in Figure~\ref{fig:DynkinD4}, and 
let $\Delta=\{\alpha_1,\alpha_2,\alpha_3,\alpha_4\}$ be the set of simple roots of the corresponding root system.
We denote by $\Phi=\Phi^+\sqcup \Phi^-$ the set of roots partitioned into positive and negative roots, and by $\Phi_{\geq -1}=-\Delta \sqcup \Phi^+$ the set of almost positive roots.   

\begin{figure}[h]
		\begin{overpic}[scale=1]{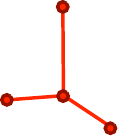}
			\put(55,35){\small $\tau_2$}
			\put(90,0){\small $\tau_1$}
			\put(55,90){\small $\tau_3$}
			\put(-25,20){\small $\tau_4$}
			
		\end{overpic}
\caption{The Dynkin diagram of type $D_4$}
\label{fig:DynkinD4}
\end{figure}

\begin{figure}[htb]
\begin{center}
\includegraphics[scale=0.9]{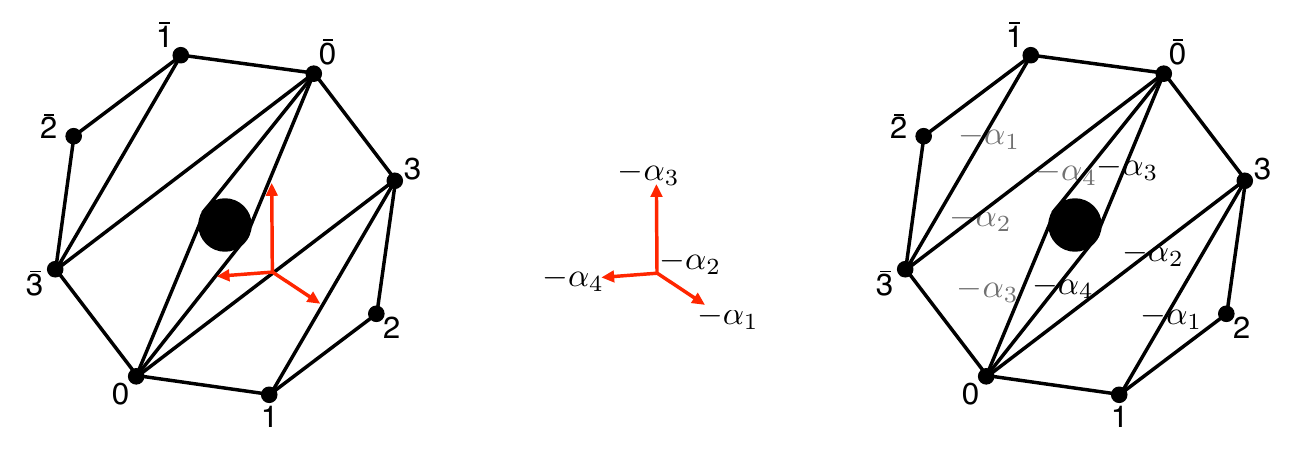}
\caption{The snake pseudotriangulation $T_0$ (left), its corresponding quiver of type~$D_4$~(middle), Snake chords labeled with the negative simple roots (right)}
\label{fig:snake}
\end{center}
\end{figure}

Let $T_0=\{0^\textsc{l},0^\textsc{r},03,13\}$ be the ``snake pseudotriangulation" illustrated in Figure~\ref{fig:snake}. We label its centrally symmetric pairs of chords with the negative simple roots $\{-\alpha_1,-\alpha_2,-\alpha_3,-\alpha_4\}$ as shown, and identify any other chord $\delta$ with the positive root obtained by adding the simple roots associated to the chords of $T_0$ crossed by $\delta$. This determines a bijection between the set of almost positive roots and the set of chords in the geometric configuration.  
The cluster complex is the simplicial complex whose maximal simplices are the sets of almost positive roots corresponding to pseudotriangulations. 

\begin{example}
For instance, the chord $\overline 12$ corresponds to the positive root $\alpha_1+2\alpha_2+\alpha_3+\alpha_4$ since it crosses two chords of $T_0$ with label~$\alpha_2$ and one chord for each label $\alpha_1,\alpha_3$ and $\alpha_4$. The cluster associated to the pseudotriangulation $T=\{1^\textsc{l}, 2^\textsc{l}, \overline 02, \overline 12\}$ is $\{\alpha_2+\alpha_4, \alpha_1+\alpha_2+\alpha_4,\alpha_1+\alpha_2, \alpha_1+2\alpha_2+\alpha_3+\alpha_4\}$.
To lighten the notation, sometimes we write~$\alpha_I$ for the positive root $\alpha_I=\sum_{i\in I}\alpha_i$ where $I$ is a multi-set with possible repeated subindices. With this notation, the cluster corresponding to $T$ is $\{\alpha_{24},\alpha_{124},\alpha_{12},\alpha_{12234}\}$. The correspondence between all chords and the positive roots is illustrated in Figure~\ref{fig:F_36_fig1} (right) on page~\pageref{fig:F_36_fig1}, and explicitly presented in Figure~\ref{fig:bijection} on page~\pageref{fig:bijection}.
\end{example}

\subsection{Combinatorial types of clusters of type $D_4$}\label{ssec:combi_type}
The cluster algebra of type $D_n$ induces a natural \Dfn{rotation} operation, denoted by $\tau$, on clusters, see \eg \cite[Section~2.2]{ceballos_denominator_2015}.
In the pseudotriangulation model, the action of $\tau$ corresponds to counterclockwise rotation by $\pi/n$, and exchanging $p^\textsc{l}$ with $p^\textsc{r}$, see~\cite{CP_pseudotriangulations}.
For instance, the chord $03$ in Figure~\ref{fig:snake} is rotated to $1\overline 0$, while $0^\textsc{r}$ is rotated to $1^\textsc{l}$.

\begin{definition}
We say that two pseudotriangulations $T_1$ and $T_2$ are \Dfn{combinatorially equivalent} modulo reflections and~$\tau$-rotations if there exists a sequence of reflections of the $2n$-gon and $\tau$-rotations transforming $T_1$ into $T_2$.
Further, we say that two clusters of type $D_n$ are \Dfn{combinatorially equivalent} if their corresponding pseudotriangulations are combinatorially equivalent modulo reflections and $\tau$-rotations.
\end{definition}

\begin{proposition}
There are exactly 7 combinatorial types of clusters of type $D_4$ modulo reflections and~$\tau$-rotations.
\end{proposition}

\begin{proof}
The type $D_4$ rotation $\tau$ exchanges $p^\textsc{l}$ and $p^\textsc{r}$ after usual rotation by $\frac{\pi}{4}$. Applying two reflections we also obtain usual rotation without this special rule. Combining $D_4$ rotation and reflections we can then obtain an operation that preserves the non-central chords and exchanges $p^\textsc{l}$ and $p^\textsc{r}$. After an analysis of the 50 pseudotriangulations modulo these four operations ($D_4$ rotation, usual rotation, reflection, and exchange of central chords) we classify them into 7 combinatorial classes in Figure 5.
\end{proof}

\begin{figure}[t]
\begin{center}
\includegraphics[width=\textwidth]{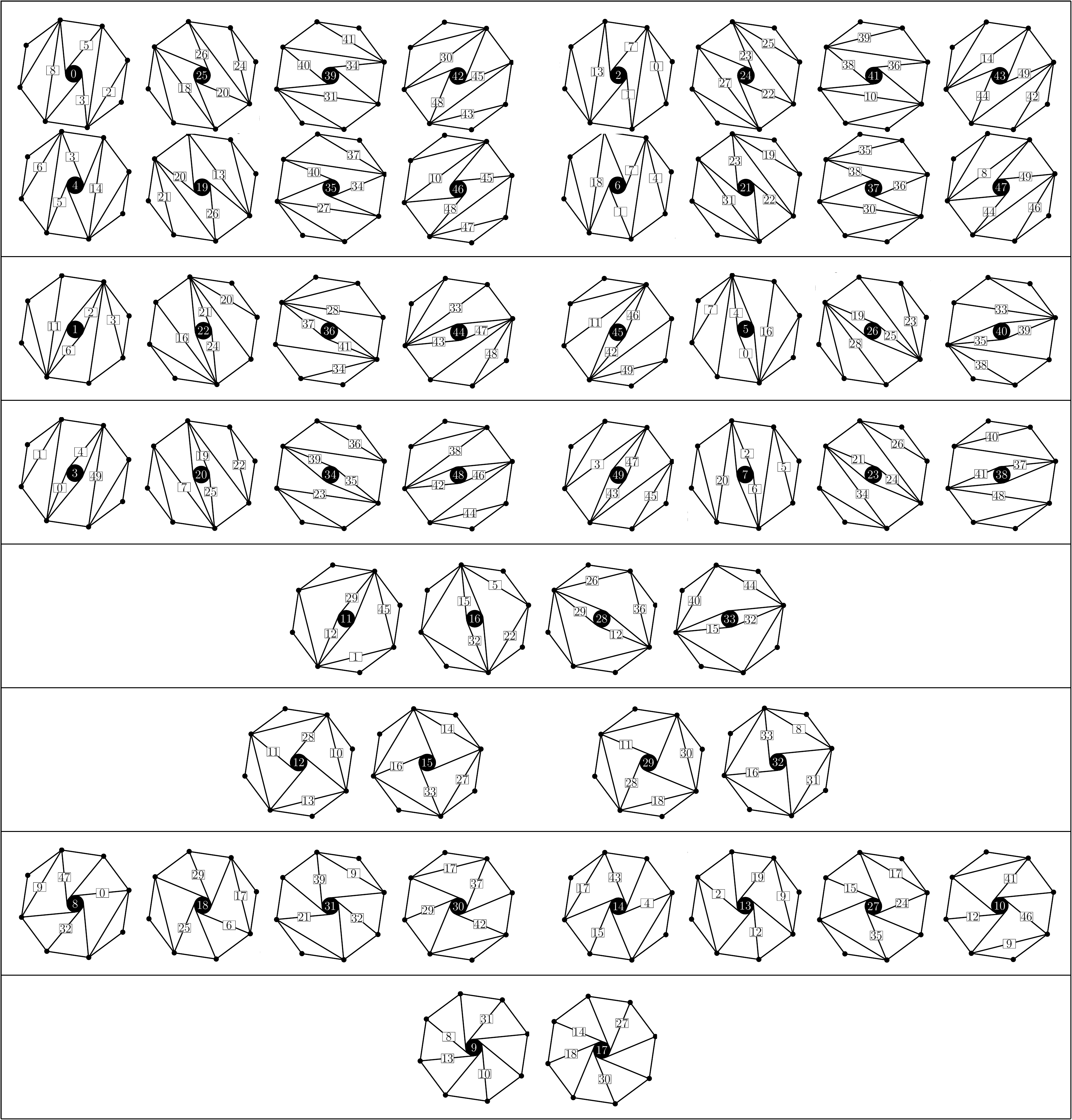}
\caption{The seven combinatorial types of pseudotriangulations (or clusters) of type~$D_4$ modulo reflections and $\sigma$-swap. The 50 pseudotriangulations are labeled with the numbers from 0 to 49 at the center of the disk. Each pair of chords in a pseudotriangulation is labeled with the pseudotriangulation obtained when flipping it.}
\label{fig:combinatorial_pseudotriang}
\vspace{-0.5cm}
\end{center}
\end{figure}

Using the fact that there are 7 combinatorial types of tropical planes in $\mathbb{TP}^5$~\cite{speyer_tropical_2004} we obtain:
\begin{corollary}
The number of combinatorial types of clusters of type $D_4$ is equal to the number of combinatorial types of tropical planes in $\mathbb{TP}^5$.
\end{corollary}

We emphasize this as a corollary as it is our main motivation for establishing a relationship between tropical planes in $\mathbb{TP}^5$ and clusters of type~$D_4$. Before making a precise connection between clusters of type $D_4$ and tropical planes in $\mathbb{TP}^5$, we recall some preliminaries in tropical geometry.

\section{Tropical Varieties and their Positive Part}
\label{sec:tropical}

The min-plus, or \emph{tropical}, semiring $\mathbb{T}:=(\mathbb{R}\cup \{-\infty\},\oplus,\odot)$ is defined with the following operations:
\[
	x\oplus y=\min\{x,y\}\: \mbox{ and }\: x\odot y=x+y,
\]
where the zero element is $\infty$ and the one element is 0. A \Dfn{tropical monomial} is an expression of the form $c\odot x_{1}^{\odot a_1}\odot \cdots \odot x_{n}^{\odot a_n}$ and a \Dfn{tropical polynomial} is a finite tropical sum of tropical monomials.
Equivalently, a tropical polynomial is a piecewise linear concave function given as the minimum of finitely many linear functions of the form $$(x_1,...,x_n)\mapsto a_1x_1+...+a_nx_n+c.$$ A \Dfn{tropical hypersurface} $\mathcal{T}(f)$ is the set of all points $\bar{x}=(x_1,...,x_n)$ in $\mathbb{R}^n$ such that the tropical polynomial $f$ is not linear at $\bar{x}$; i.e. $\mathcal{T}(f)$ is the set of points $\bar{x}$ at which the minimum of the values attained by the linear functions defining $f$ is attained by at least two of the linear functions defining $f$.\\

The tropical objects defined above are inherently tropical, in the sense that they are defined over $\mathbb{T}$. It is also possible to consider ``classical'' algebraic objects over a fixed field $\mathbb{K}$ and construct their tropical analogues via a nonarchimedean valuation.\

Let us consider the field $\mathbb{K}:=\mathbb{C}\{\{t\}\}=\{\sum \alpha_r t^{r/n} : \alpha_r\in \mathbb{C}, r\in \mathbb{Z}, n\in\mathbb{Z}_+\}$ of \Dfn{Puiseux series} over $\mathbb{C}$. The field $\mathbb{K}$ is of characteristic 0 and is algebraically closed. Our field $\mathbb{K}$ has a natural nonarchimedean valuation $\val:\mathbb{K}\rightarrow \mathbb{T}$ sending 0 to $-\infty$ and, for $a(t)\neq 0$, sending $a(t)\mapsto \min\{\frac{r}{n} : \alpha_r\neq 0\}$. Note that this valuation extends naturally to a valuation $\Val: \mathbb{K}^n\rightarrow \mathbb{T}^n$ on $\mathbb{K}^n$ via coordinatewise evaluation: $(a_1(t),...,a_n(t))\mapsto (\val(a_1(t)),\dots,\val(a_n(t))$. Let us consider the polynomial $$f=\sum_{i=1}^{k} a_i(t)x_{1}^{b_{1}^i}\cdots x_{n}^{b_{n}^i}\in \mathbb{K}[t_1,...,t_n].$$ We then define the \Dfn{tropicalization} of $f$, obtained by replacing $+$ with $\oplus$, $\cdot$ with $\odot$, and all coefficients with their valuation. There is also a notion of \Dfn{tropical projective space}, $$\mathbb{TP}^n:=\mathbb{T}^{n+1}\backslash \{(\infty,...,\infty)\},$$ which is defined by the equivalence relation $v\sim v+\lambda (1,...,1)$. In other words, $\mathbb{TP}^n$ is equal to $(\mathbb{T}^{n+1}\backslash \{(\infty,...,\infty)\})/(1,...,1)$. We then have that the coordinate system $(x_1,...,x_{n+1})$ on $\mathbb{T}^{n+1}$ induces a tropical homogeneous coordinate system $[x_1:...:x_{n+1}]$ on $\mathbb{TP}^n$ given by the natural embedding, $\mathbb{R}^n\rightarrow \mathbb{TP}^{n+1},\mbox{  } (x_1,...,x_n)\mapsto [x_1:...:x_n:0]$.\\

Picking a $w\in \mathbb{R}^n$, the \Dfn{$w$-weight} of a term $a_i(t)x_{1}^{b_{1}^i}\cdots x_{n}^{b_{n}^i}$ of the polynomial $f$ is equal to the dot product $\langle \val(a_i(t))+ w, \bar{b}\rangle$, where $\bar{b}=(b_{1}^i,...,b_{n}^i)$. We then define $\ini_{w}(f)$ to be the polynomial consisting of the sum of the terms of $f$ with the largest $w$-weight. A \Dfn{tropical hypersurface} can then be defined as $$\mathcal{T}(f) :=\{w\in \mathbb{R}^n:\ \ini_{w}(f)\ \mbox{is not a monomial}\}.$$ It is straightforward to check that the ``inherently tropical'' definition of a tropical hypersurface given at the beginning of this section is equivalent to the one directly above. A \Dfn{root} of a tropical polynomial $f$ in~$n$ variables $x_1,...,x_n$ is a point $(b_1,...,b_n)\in \mathbb{R}^n$ such that the value $f(b_1,...,b_n)$, is attained by least two of the linear functions defining $f$.\

A finite intersection of tropical hypersurfaces is called a \Dfn{tropical prevariety}. Given an ideal $I\in \mathbb{K}[t_1,...,t_n]$ and $w\in \mathbb{R}^n$, we define the \Dfn{tropical variety} $\mathcal{T}(I)$ to be the intersection of tropical hypersurfaces $$\mathcal{T}(I) :=\bigcap_{f\in I} \mathcal{T}(f)$$ along with a weight function $\Omega_w: \mathbb{K}[x_1,...x_n]\rightarrow \mathbb{R}$, sending $$a_i(t)x_{1}^{b_{1}^i}\cdots x_{n}^{b_{n}^i}\mapsto \langle \val(a_i(t))+ w, \bar{b}\rangle$$ for each monomial of a polynomial $f=\sum_{i=1}^{k} a_i(t)x_{1}^{b_{1}^i}\cdots x_{n}^{b_{n}^i}\in I$, where $\bar{b}=(b_{1}^i,...,b_{n}^i)$. Equivalently, the tropical variety $\mathcal{T}(I)$ can be defined as $$\mathcal{T}(I) :=\{w\in \mathbb{R}^n: \ini_{w}(I)\ \mbox{has no monomials}\},$$ where $\ini_{w}(I)$ is the ideal generated by the set $\{\ini_{w}(f): f\in I\}$. It is important for us to note that if we have a variety $V(I)$, we can consider the closure of the image of $V(I)$ under the map $$\trop : \mathbb{K}^n\rightarrow \mathbb{R}^n, x\mapsto \val(x)$$ and that we get $\mathcal{T}(I)=\overline{\trop(V(I))}$.

\subsection{The positive part of a tropical variety}

Now that we have some idea as to what a tropical variety is, we can define the positive part of a tropical variety. The \Dfn{positive part} of a tropical variety $\mathcal{T}(I)$ is defined to be
\[
	\mathcal{T}^+(I):=\overline{trop(V(I)\cap (\mathbb{K}^+)^n)},
\]
where $\mathbb{K}^+:=\{ \alpha(t)\in \mathbb{K} : \mbox{the coefficient of the lowest term of } \alpha(t) \mbox{ is real and positive}\}$. The following theorems are useful to uncover the positive part of a tropical variety:

\begin{theorem}[Speyer and Williams \cite{speyer_tropical_2005}]
A point $w=(w_1,...,w_n)$ lies in $\mathcal{T}^+(I)$ if and only if $\ini_w(I)$ does not contain any nonzero polynomials in $\mathbb{K}^+[t_1,...,t_n]$.
\end{theorem}

\begin{theorem}[Speyer and Williams \cite{speyer_tropical_2005}]
An ideal $I$ of $\mathbb{K}[t_1,...,t_n]$ contains a nonzero element of $\mathbb{K}^+[t_1,...,t_n]$ if and only if $(\mathbb{K}^+)^n\cap V(\ini_w(I))=\emptyset$, for all $w\in \mathbb{K}^n$.
\end{theorem}

In practice, the positive part of a tropical variety can be uncovered by identifying each domain of linearity with the sign of its defining monomial and taking the components of the tropical variety which separate regions of different signs to be the positive part. Let us see what this means through an example:

\begin{example}
Let us consider the polynomial $1-x-y+txy=0$. Its tropicalization $$f=\min\{0,x,y,x+y-1\}$$ can be represented as a partition of the plane into polygonal regions, as illustrated in Figure~\ref{fig:tropvar}. We want to find solutions in $\mathbb{K}^+$ to $1-x-y+txy=0$ and take the closure of the tropicalization of this to get the positive part. Say $(x,y)=(a_0t^{b_0}+\cdots, c_0t^{d_0}+\cdots )$ is such a solution, where $a_0$ and $c_0$ are the lowest terms. We then have that $$1-(a_0t^{b_0}+\cdots )-(c_0t^{d_0}+\cdots ) +t^{-1}(a_0t^{b_0}+\cdots )(c_0t^{d_0}+\cdots )=0,$$ and we can see that the only terms which have even the possibility of cancelling each other out are terms with different signs (as $a_0$ and $c_0$ are real and positive). So the positive part of our tropical variety is composed of the components of our tropical variety separating linear regions defined by monomials of different signs, as seen in Figure~\ref{fig:tropvar}.

We can then see that finding the positive part of a tropical variety by identifying linear regions with the sign of their defining linear term works as a method for finding the positive part of a tropical variety in general, meaning that \emph{the positive part of a tropical variety recaptures the sign of each monomial, which is initially lost through the tropicalization process}.

For this example specifically, we can identify the region in which $f$ is equal to $x+y-1$ with a ``$+$'' sign, since the sign of $txy$ is ``$+$''. Similarly, we identify the region in which $f$ is equal to $x$ and the region in which $f$ is equal to $y$ with ``$-$'' signs and the region in which $f$ is $0$ with a ``$+$''. The positive part is the subset of our tropical variety below defined by the components in bold as opposed to dashed.

\begin{figure}[!tbp]
\begin{center}
\begin{tabular}{c@{\hspace{2cm}}c}
\begin{tikzpicture}[fleche/.style={<-,line width=2pt,>=stealth, cap=round,blue},
                    ligne/.style={line width=2pt,cap=round,blue}]
\def\dim{1}
\draw (-2*\dim,0) -- coordinate (x axis mid) (\dim,0);
\draw (0,-2*\dim) -- coordinate (y axis mid) (0,\dim);
\foreach \x in {-2,-1,1}
    \draw (\x*\dim,1pt) -- (\x*\dim,-3pt)
    node[anchor=north] {\x};
\draw (0,1pt) -- (0,-3pt) node[anchor=north west] {0};
\foreach \y in {-2,-1,1}
    \draw (1pt,\y*\dim) -- (-3pt,\y*\dim)
    node[anchor=east] {\y};
\draw (1pt,0) -- (-3pt,0) node[anchor=south east] {0};

\draw[fleche] (1.5*\dim,0) -- (0,0);
\draw[fleche] (0,1.5*\dim) -- (0,0);
\draw[ligne] (0,0) -- (-\dim,-\dim);
\draw[fleche] (-\dim,-2.5*\dim) -- (-\dim,-\dim);
\draw[fleche] (-2.5*\dim,-\dim) -- (-\dim,-\dim);

\node at (\dim,\dim) {$\mathbf{f=0}$};
\node at (-1.25*\dim,1*\dim) {$\mathbf{f=x}$};
\node at (1*\dim,-1*\dim) {$\mathbf{f=y}$};
\node at (-2.25*\dim,-1.5*\dim) {$\mathbf{f=x+y-1}$};
\end{tikzpicture}
&
\begin{tikzpicture}[fleche/.style={<-,line width=2pt,>=stealth, cap=round,blue},
                    ligne/.style={line width=2pt,cap=round,blue},
                    lignedashed/.style={dashed,line width=1pt,cap=round}]

\def\dim{1}
\draw (-2*\dim,0) -- coordinate (x axis mid) (\dim,0);
\draw (0,-2*\dim) -- coordinate (y axis mid) (0,\dim);
\foreach \x in {-2,-1,1}
    \draw (\x*\dim,1pt) -- (\x*\dim,-3pt)
    node[anchor=north] {\x};
\draw (0,1pt) -- (0,-3pt) node[anchor=north west] {0};
\foreach \y in {-2,-1,1}
    \draw (1pt,\y*\dim) -- (-3pt,\y*\dim)
    node[anchor=east] {\y};
\draw (1pt,0) -- (-3pt,0) node[anchor=south east] {0};

\draw[fleche] (1.5*\dim,0) -- (0,0);
\draw[fleche] (0,1.5*\dim) -- (0,0);
\draw[lignedashed] (0,0) -- (-\dim,-\dim);
\draw[fleche] (-\dim,-2.5*\dim) -- (-\dim,-\dim);
\draw[fleche] (-2.5*\dim,-\dim) -- (-\dim,-\dim);

\node at (0.5*\dim,0.5*\dim) {\Large$\mathbf{+}$};
\node at (-1*\dim,1*\dim) {\Large$\mathbf{-}$};
\node at (1*\dim,-1*\dim) {\Large$\mathbf{-}$};
\node at (-1.5*\dim,-1.5*\dim) {\Large$\mathbf{+}$};
\end{tikzpicture}
\end{tabular}
\end{center}
\caption{The tropical variety defined by $f=\min\{0,x,y,x+y-1\}$ on the left, and its positive part on the right}
\label{fig:tropvar}
\end{figure}
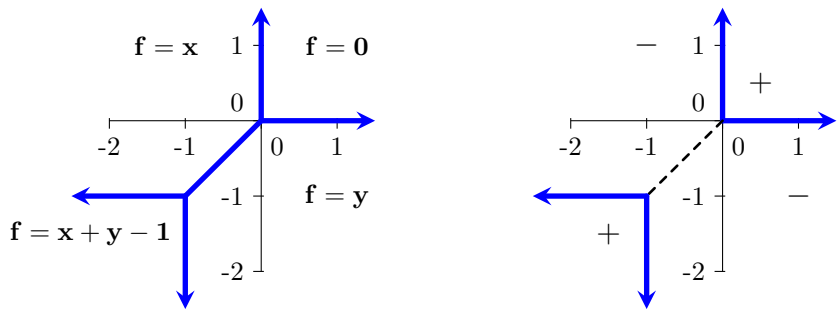
\end{example}

\subsection{The Dressian}

As discussed in \cite{speyer_tropical_2008}, a \Dfn{tropical linear space} is a tropical variety $\mathcal{T}(I)$ given by an ideal of the form $$I=\langle a_{i1}x_1+a_{i2}x_2+\cdots +a_{in}x_n : i=1,2,...,n-d\rangle,$$ where the $a_{ij}$ are the entries of any $(n-d)\times n$-matrix of rank $n-d$ with entries in $\mathbb{K}[t_1,...,t_n]$. 
The ideal $I$ can be rewritten in terms of a vector in $\mathbb{R}^{{n \choose d}}$ whose entries are the \Dfn{Pl\"{u}cker coordinates}
	$$p_{i_1i_2\cdots i_d} := (-1)^{i_1+i_2+\cdots +i_d}\cdot \mbox{det} 
	\left(\begin{array}{cccc}
		a_{1,j_1} & a_{1,j_2} & \cdots & a_{1,j_{n-d}} \\
		a_{2,j_1} & a_{2,j_2} & \cdots & a_{2,j_{n-d}} \\
		\vdots & \vdots & \ddots & \vdots \\
		a_{n-d,j_1} & a_{n-d,j_2} & \cdots & a_{n-d,j_{n-d}}
	\end{array}\right),$$
where $i_1<\cdots <i_d, j_1<\cdots < i_{n-d}$ and $\{i_1,...,i_d,j_1,...,j_{n-d}\}=\{1,...,n\}$. Explicitly, we have
\[
	I=\langle \sum_{r=0}^d (-1)^r\cdot p_{i_0 i_1\cdots \hat{i_r}\cdots i_d}\cdot x_{i_r} : \text{ for all } 1\leq i_0\leq i_1<\cdots <i_r\leq n\rangle.
\]
Speyer gives the details on how to derive $I$ in terms of Pl\"{u}cker coordinates in \cite{speyer_tropical_2008}. The \Dfn{Dressian} $\Dr(d,n)$ is the tropical prevariety which parametrizes ($d-1$)-dimensional tropical linear spaces, and is more explicitly the tropical prevariety consisting of the intersection of the tropical hypersurfaces given by all 3-term Pl\"{u}cker relations. The Dressian can also be defined as the polyhedral fan of those regular subdivisions of the \textit{(d,n)-hypersimplex} which have the property that each cell is a \textit{matroid polytope} \cite{herrmann_dressians_2012}. More specifically, let $e_1,...,e_n$ be the standard basis vectors for $\mathbb{R}^n$ and let $e_X :=\sum_{i\in X} e_i$ for any subset $X\subseteq [n]$. Given a matroid $X\subseteq {[n] \choose d}$, we define its \Dfn{matroid polytope} to be the polytope $P_X :=\conv \{e_x : x\in X\}.$ The \Dfn{$(d,n)$-hypersimplex} $\triangle (d,n)$ in $\mathbb{R}^n$ is defined as $\triangle (d,n) := P_{{[n] \choose d}}.$ A \Dfn{matroid subdivision} of a polytope $P$ is a polytopal subdivision of $P$ such that each of its cells is a matroid polytope. A \Dfn{weight vector} $\lambda$ on an polytope $P$ assigns a real number to each vertex of~$P$. A given weight vector $\lambda$ induces a polytopal subdivision of a polytope $P$ by considering the set $\conv \{(v,\lambda (v)) : v \mbox{ vertex of } P \}$ in $\mathbb{R}^{n+1}$ and projecting the lower (or upper) envelope to the hyperplane $(\mathbb{R}^n,0)$; a polytopal subdivision of this kind is called a \Dfn{regular} polytopal subdivision.
The Dressian is shown in \cite{herrmann_dressians_2012} to be a subfan of the secondary fan of $\triangle (d,n)$:

\begin{proposition}\label{proposition 3.1}
\textup{(Herrmann, Jensen, Joswig, and Sturmfels \cite[Proposition 3.1]{herrmann_how_2009})} A \textit{weight vector} $\lambda\in \mathbb{R}^{{[n] \choose d}}$ lies in the Dressian $\Dr(d,n)$, seen as a fan, if and only if it induces a \textit{matroid subdivision} of the hypersimplex $\triangle (d,n)$.
\end{proposition}

It is pointed out in \cite{herrmann_how_2009} that the three term Pl\"{u}cker relations define a natural Pl\"{u}cker fan structure on the Dressian: two weight vectors $\lambda$ and $\lambda'$ are in the same cone if they specify the same initial form for each 3-term Pl\"{u}cker relation.\\

\begin{remark}
Tropical Pl\"{u}cker vectors (i.e. vectors of tropicalizations of Pl\"{u}cker coordinates) can be viewed as valuated matroids \cite[Remark 2.4]{herrmann_how_2009}. A \Dfn{valuated matroid} of rank $d$ on a set $[n]$ is a map $\pi: [n]^d\rightarrow \mathbb{R}\cup \{\infty\}$ such that 
\begin{enumerate}
\item $\pi(\omega)$ is independent of the ordering of the sequence vector $\omega$,
\item $\pi(\omega)=\infty$ if an element occurs twice in $\omega$,
\item for every ($d-1$)-subset $\sigma$ and every ($d+1$)-subset $\tau = \{\tau_1,\tau_2,...,\tau_{d+1}\}$ of $[n]$ the minimum of $$\pi(\sigma\cup \{\tau_i\}) + \pi(\tau\backslash \{\tau_i\}) \mbox{ for } 1\leq i\leq d+1$$ is attained at least twice.
\end{enumerate}
\end{remark}

\textit{Tropical planes} are dual to regular matroid subdivisions of the hypersimplex $\triangle (3,n)$, thus giving us another way to view $\Dr(3,n)$: The parameter space of tropical planes. A \Dfn{tropical plane} $L_p$ in $\mathbb{TP}^{n-1}$, for some $p\in \Dr(3,n)$, is the intersection of the tropical hyperplanes $$\bigcap_{\{i,j,k,l \}\in {[n] \choose 4}} \mathcal{T}(p_{ijk}x_l+p_{ijl}x_k+p_{ikl}x_j+p_{jkl}x_i),$$ where the Pl\"{u}cker coefficients appearing in $p_{ijk}x_l+p_{ijl}x_k+p_{ikl}x_j+p_{jkl}x_i$ are entries of $p$ lexicographically indexed by the order $i<j<k<l$.\

Just as planes in projective space $\mathbb{P}^{n-1}$ correspond to arrangements of $n$ lines in $\mathbb{P}^2$, tropical planes in $\mathbb{TP}^{n-1}$ correspond to arrangements of $n$ \textit{tropical lines} in $\mathbb{TP}^2$. A \Dfn{tropical line} in $\mathbb{TP}^{n-1}$ is an embedded metric tree which is balanced and has $n$ unbounded edges pointing in the coordinate directions. Thus, we can use arrangements of trees to represent matroid subdivisions of $\triangle (3,n)$. We say trees, and not metric trees due to the following result:

\begin{proposition}
\textup{(Herrmann, Jensen, Joswig, and Sturmfels \cite[Proposition 4.1]{herrmann_how_2009})} Each metric tree arrangement gives rise to an abstract tree arrangement by ignoring the edge lengths.
\end{proposition}

We then have that:

\begin{lemma}
\textup{(Herrmann, Jensen, Joswig, and Sturmfels \cite[Lemma 4.2]{herrmann_how_2009})} Each matroid subdivision $\Sigma$ of $\triangle (3,n)$ defines an abstract arrangement $T(\Sigma)$ of $n$ trees. Moreover, if $\Sigma$ is regular then $T(\Sigma)$ supports a metric tree arrangement.
\end{lemma}

For a definition of \textit{abstract tree arrangement}, we refer the reader to \cite[Section 4]{herrmann_how_2009}. The bijection between tropical planes and arrangements of metric trees is studied in great detail in \cite{herrmann_how_2009}, with their main theorem being:

\begin{theorem}
\textup{(Herrmann, Jensen, Joswig, and Sturmfels \cite[Theorem 4.4]{herrmann_how_2009})} The equivalence classes of arrangements of $n$ metric trees are in bijection with regular matroid subdivision of the hypersimplex $\triangle (3,n)$. Moreover, the secondary fan structure on $\Dr(3,n)$ coincides with the Pl\"{u}cker fan structure.
\end{theorem}

\subsection{The tropical Grassmannian and its positive part}
The \Dfn{tropical Grassmannian} $\Gr(d,n)$ is a tropical variety which is a subset of the Dressian $\Dr(d,n)$. As fans, the Grassmannian and the Dressian have the same $n$-dimensional lineality space and thus can be viewed as pointed fans in $\mathbb{R}^{{n \choose d}-n}$, one sitting inside of the other. Explicitly, the tropical Grassmannian is $\mathcal{T}(I_{d,n})$, where $I_{d,n}$ \Dfn{Pl\"{u}cker ideal}; i.e. the ideal generated by the Pl\"{u}cker relations.\

The tropical Grassmannian was first studied in Speyer and Sturmfels
\cite{speyer_tropical_2004} and its positive part was then studied in Speyer
and Williams \cite{speyer_tropical_2005}. Speyer and Williams lay out the first
steps in studying the positive part of a tropical variety and explicitly
outline a way of parametrizing the positive part $\Gr^+(d,n)$ of $\Gr(d,n)$
using a particular kind of directed graph $\Web_{d,n}$, which is a special case
of the $\reflectbox{L}$-diagrams of Postnikov \cite{postnikov_total_2006}. The short story of their parametrization is that they develop a bijection $\Phi_2$ from $(\mathbb{R}^+)^{(d-1)(n-d-1)}$ to the real, positive points of the Grassmannian (modulo its lineality space) using $\Web_{d,n}$. 
They tropicalize this map to get a surjection from $\mathbb{R}^{(d-1)(n-d-1)}$ to the positive part of the tropical Grassmannian (modulo its lineality space). 
Then they define a complete fan $F_{d,n}$ in $\mathbb{R}^{(d-1)(n-d-1)}$ whose maximal cones are the domains of linearity of the tropicalization of $\Phi_2$.\

For $\Gr^+(2,n)$, Speyer and Williams~\cite{speyer_tropical_2005} 
show that their fan $F_{2,n}$ is equal to the Stanley--Pitman fan $F_{n-3}$ of \cite{stanley_polytope_2002}, which is combinatorially isomorphic to the cluster complex of type $A_{n-3}$. 
They also show that the tropical Grassmannians $\Gr^+(3,6)$ and $\Gr^+(3,7)$ are closely related to the cluster complexes of type $D_4$ and $E_6$ respectively. The connection between $\Gr^+(3,6)$ and the cluster complex of type~$D_4$ will be made precise below.

\section{Connecting the cluster complex of type $D_4$ to $\Gr^+(3,6)$}
\label{sec:connecting}

In~\cite[Proposition~6.1]{speyer_tropical_2005}, Speyer and Williams provide an explicit computation of the fan $F_{3,6}$ associated to the tropical Grassmannian~$\Gr^+(3,6)$, together with inequalities defining a polytope that~$F_{3,6}$ is normal to. They computed the $f$-vector $(16,66,98,48)$ of $F_{3,6}$ and noticed that it is very close to the $f$-vector $(16,66,100,50)$ of the cluster complex of type~$D_4$. They found that $F_{3,6}$ has two cones which are of the form of a cone over a bipyramid, and stated that when subdividing these two bipyramids into two tetrahedra each, one gets a fan that is combinatorially isomorphic to the cluster complex of type~$D_4$.   

\begin{figure}[!htbp]
		\begin{overpic}[width=0.75\textwidth]{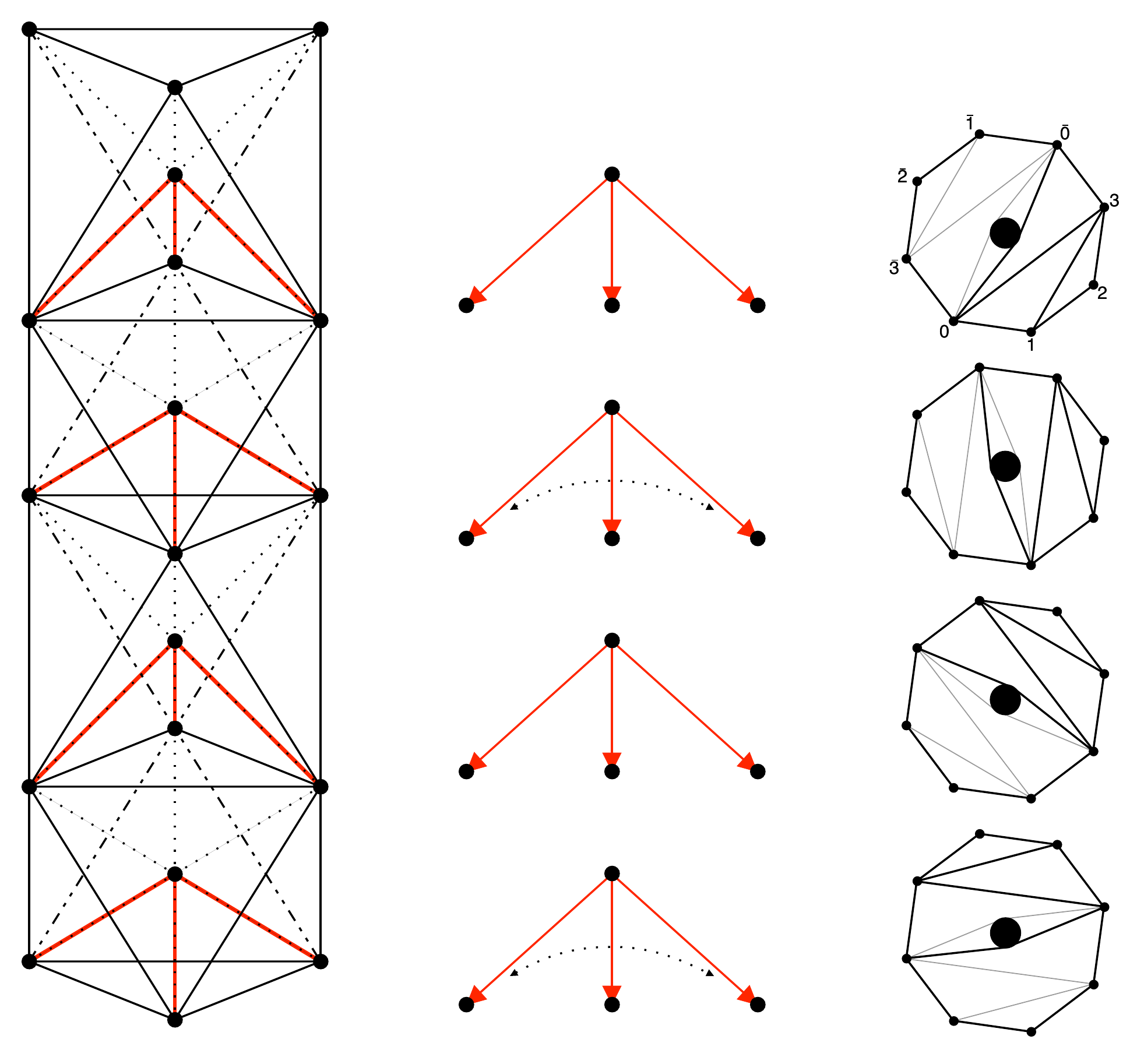}
			\put(-0.5,8){\small $r_9$}
			\put(29.5,8){\small $r_3$}
			\put(17,3){\small $r_{12}$}
			\put(17,16.5){\small $r_{10}$}

			\put(-0.5,49){\small $r_{14}$}
			\put(29.5,49){\small $r_6$}
			\put(17,44){\small $r_2$}
			\put(17,57.5){\small $r_7$}

			\put(-0.5,24){\small $r_8$}
			\put(29.5,24){\small $r_4$}
			\put(17,30){\small $r_{15}$}
			\put(17,37){\small $r_{16}$}

			\put(-0.5,65){\small $r_5$}
			\put(29.5,65){\small $r_{11}$}
			\put(17,71){\small $r_1$}
			\put(17,78){\small $r_{13}$}
			
			\put(40,3){\small $\alpha_3$}
			\put(52.5,3){\small $\alpha_1$}
			\put(65.5,3){\small $\alpha_4$}
			\put(52.5,18){\small $\alpha_{1234}$}
			
			\put(40,23.5){\small $\alpha_{124}$}
			\put(52.5,23.5){\small $\alpha_{234}$}
			\put(65.5,23.5){\small $\alpha_{123}$}
			\put(52.5,38.5){\small $\alpha_{12234}$}

			\put(40,44){\small $\alpha_{23}$}
			\put(52.5,44){\small $\alpha_{12}$}
			\put(65.5,44){\small $\alpha_{24}$}
			\put(52.5,59){\small $\alpha_2$}

			\put(40,64.5){\small $-\alpha_3$}
			\put(52.5,64.5){\small $-\alpha_1$}
			\put(65.5,64.5){\small $-\alpha_4$}
			\put(52.5,79.5){\small $-\alpha_2$}

			\put(84,8.5){\scriptsize $\alpha_3$}
			\put(89,17){\scriptsize $\alpha_1$}
			\put(93,11){\scriptsize $\alpha_4$}
			\put(83,14){\scriptsize $\alpha_{1234}$}
			
			\put(82,33.5){\scriptsize $\alpha_{124}$}
			\put(91.5,35){\scriptsize $\alpha_{234}$}
			\put(90,28.5){\scriptsize $\alpha_{123}$}
			\put(83.3,36){\scriptsize $\alpha_{12234}$}

			\put(83.5,56){\scriptsize $\alpha_{23}$}
			\put(92,51){\scriptsize $\alpha_{12}$}
			\put(86,47){\scriptsize $\alpha_{24}$}
			\put(90,56){\scriptsize $\alpha_2$}

			\put(91,76){\scriptsize $-\alpha_3$}
			\put(91.5,68){\scriptsize $-\alpha_1$}
			\put(83,69){\scriptsize $-\alpha_4$}
			\put(89,70){\scriptsize $-\alpha_2$}

		\end{overpic}
	\caption{Speyer--Williams fan $F_{3,6}$ (part 1) and the bijection between its rays, almost positive roots, and centrally symmetric pairs of chords in the geometric model}
	\label{fig:F_36_fig1}
\end{figure}

\begin{figure}[!thbp]
		\begin{overpic}[width=0.6\textwidth]{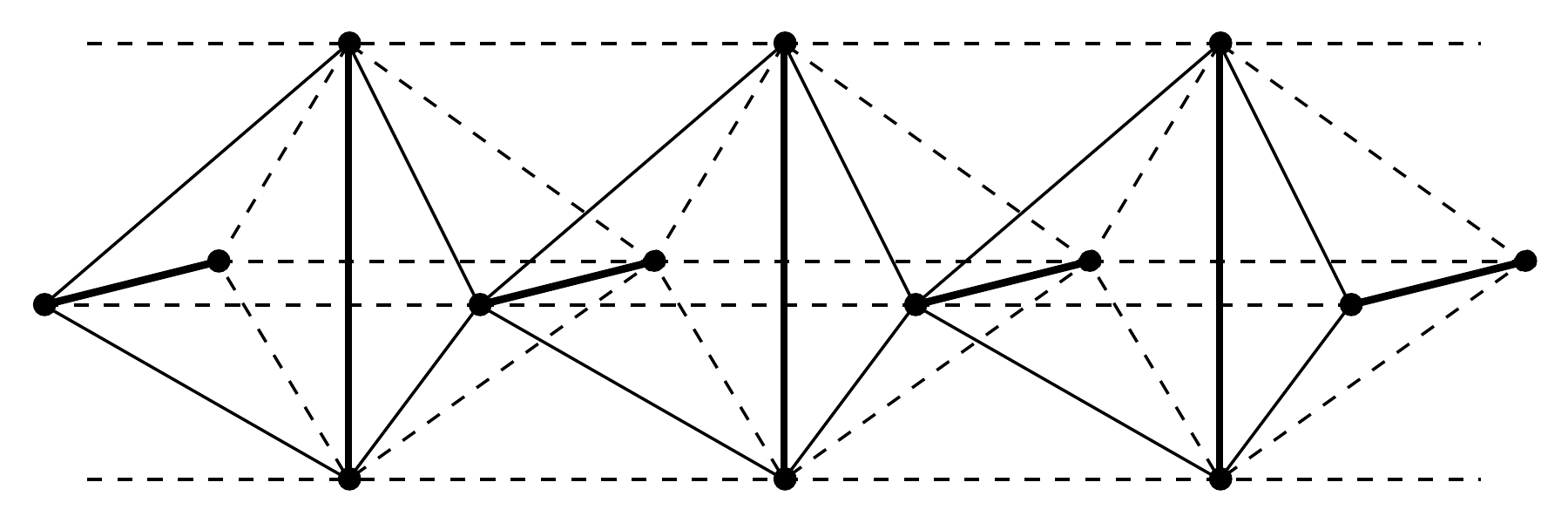}
			\put(21,32){\small $r_5$}
			\put(21,0){\small $r_8$}
			\put(48.8,32){\small $r_1$}
			\put(48.8,0){\small $r_{15}$}
			\put(76.7,32){\small $r_{11}$}
			\put(76.7,0){\small $r_4$}

			\put(16,17.5){\small $r_{2}$}
			\put(1.3,10.8){\small $r_{12}$}
			\put(43.8,17.5){\small $r_{14}$}
			\put(29.1,10.8){\small $r_{9}$}
			\put(71.6,17.5){\small $r_{6}$}
			\put(56.9,10.8){\small $r_{3}$}
			\put(99,17.5){\small $r_{2}$}
			\put(84.7,10.8){\small $r_{12}$}

		\end{overpic}
	\caption{Speyer--Williams fan $F_{3,6}$ (part 2)}
	\label{fig:F_36_fig2}
\end{figure}
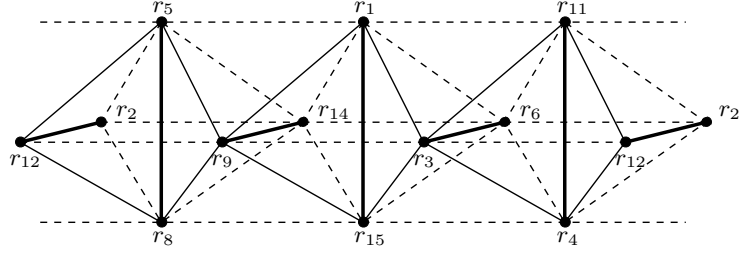

In this section, we make this connection more precise by providing an explicit bijection between the rays of $F_{3,6}$ and the almost positive roots of type $D_4$. The cones of the fan correspond to clusters, with the exception of the two cones over a bipyramid, which correspond to two clusters each, glued together on their common face. 

The fan $F_{3,6}$ is a 4-dimensional fan whose intersection with a 3-sphere is illustrated in Figure~\ref{fig:F_36_fig1} (left) and Figure~\ref{fig:F_36_fig2}.
Each of the figures is a solid torus, and the two figures glue together to form a 3-sphere. The two bipyramids have vertices $\{r_1,r_5,r_7,r_{11},r_{13}\}$ and $\{r_4,r_8,r_{10},r_{15},r_{16}\}$.
These figures are reproduced from Speyer and Williams original figures~\cite[Figures~7 and~8]{speyer_tropical_2005}. In addition, we include part of the Auslander--Reiten quiver of type $D_4$, formed by repeating 4 copies of a bipartite quiver of type $D_4$. The 16 vertices of this ``repetition quiver" are labelled by the 16 almost positive roots as shown in Figure~\ref{fig:F_36_fig1} (middle). We also include the corresponding 16 pairs of chords in the geometric model in Figure~\ref{fig:F_36_fig1} (right).

The labeling with almost positive roots in Figure~\ref{fig:F_36_fig1} (middle) can be explained in two different ways. The first, and perhaps more intuitive one, assigns the negative simple roots to the vertices of the first copy of the $D_4$ quiver on the top, and the other labels are determined by rotation. Rotation sends a vertex in a copy of a $D_4$ quiver to the same vertex in the next copy directly below, if any. The last copy of the~$D_4$ quiver in the bottom is rotated to the first copy on the top. The rotation on almost positive roots is the one induced by rotation of chords in the geometric model. Recall that this is given by counterclockwise rotation by $\pi/4$ together with the special rule of exchanging central chords~$p^\textsc{r}$ and~$p^\textsc{l}$. A chord $\delta$ not in the initial snake is labeled by the positive root obtained by adding the simple roots corresponding to the chords of the snake that are crossed by $\delta$. Figure~\ref{fig:F_36_fig1} (right) illustrates this rotation process together with the root labeling of the chords. Note that rotating one more time the chords in the bottom picture recovers back the initial snake triangulation.   
The second explanation of the labeling by almost positive roots can be done in terms of inversions of a word $P=\tau_2 \tau_1 \tau_3 \tau_4 | \tau_2 \tau_1 \tau_3 \tau_4 | \tau_2 \tau_1 \tau_3 \tau_4 $. The word~$P$ is a reduced expression for the longest element of the Coxeter group and its inversions give all positive roots. Moreover, it consists of three copies of $\tau_2 \tau_1 \tau_3 \tau_4$, and its letters are in correspondence with the vertices of the last three copies of the bipartite quiver of type $D_4$ in Figure~\ref{fig:F_36_fig1} (middle). The labeling assigns to the vertices of these last three copies the inversions of $P$, and to the vertices of the first copy of the $D_4$ quiver the negative simple roots. This second explanation is based in work on subword complexes in~\cite{ceballos_subword_2014}, we refer to~\cite[Section~2.2]{ceballos_denominator_2015} for a concise and more detailed presentation. 

Let $\Psi$ be the bijection from the rays of Speyer--Williams fan $F_{3,6}$ and almost positive roots given in Figure~\ref{fig:bijection} (left and middle).

\begin{figure}[!htbp]
\begin{center}
\resizebox{\hsize}{!}{
\begin{tabular}{cc}
$\begin{array}{c@{\hspace{0.5cm}\longleftrightarrow\hspace{0.5cm}}c@{\hspace{0.5cm}\longleftrightarrow\hspace{0.5cm}}c}
\text{Rays of } F_{3,6} & \Phi_{\geq-1} \text{ of type } D_4  & \text{Chords}\T\B\\\hline
r_1= (0,0,1,0) & -\alpha_1 & 13\\
r_2= (0,0,-1,0) & \alpha_1 + \alpha_2 & 0\bar{2}\\
r_3= (1,0,0,0) & \alpha_3 & 3^\textsc{r}\\
r_4= (1,0,-1,0) & \alpha_1 +\alpha_2+\alpha_3 & 2^\textsc{r}\\
r_5= (-1,0,0,0) & -\alpha_3 & 0^\textsc{l}\\
r_6= (0,0,0,1) & \alpha_2+\alpha_3 & 1^\textsc{r}\\
r_7= (-1,0,0,1) & \alpha_2 & 0\bar{1} \\
r_8= (0,0,0,-1) & \alpha_1 + \alpha_2 + \alpha_4 & 2^\textsc{l}\\
\end{array}$
&
$\begin{array}{c@{\hspace{0.5cm}\longleftrightarrow\hspace{0.5cm}}c@{\hspace{0.5cm}\longleftrightarrow\hspace{0.5cm}}c}
\text{Rays of } F_{3,6} & \Phi_{\geq-1} \text{ of type } D_4  & \text{Chords}\T\B\\\hline
r_9= (0,0,1,-1) & \alpha_4 & 3^\textsc{l}\\
r_{10}= (1,0,0,-1) & \alpha_1+\alpha_2+\alpha_3+\alpha_4 & 2\bar{3}\\
r_{11}= (0,1,0,0) & -\alpha_4 & 0^\textsc{r}\\
r_{12}= (0,1,0,-1) & \alpha_1 & 02\\
r_{13}= (0,1,1,-1) & -\alpha_2 & 03\\
r_{14}= (0,-1,0,0) & \alpha_2+\alpha_4 & 1^\textsc{l}\\
r_{15}= (1,-1,0,0) & \alpha_2+\alpha_3+\alpha_4 & 1\bar{3}\\
r_{16}= (1,-1,-1,0) & \alpha_1+2\alpha_2+\alpha_3+\alpha_4 & 1\bar{2}
\end{array}$
\end{tabular}}
\caption{A bijection from the rays of $F_{3,6}$ to the almost positive roots of type $D_4$ and to centrally symmetric pairs of chords in the geometric model}
\label{fig:bijection}
\end{center}
\end{figure}

This bijection sends the vertices of the four bold $D_4$ quiver in Figure~\ref{fig:F_36_fig1}~(left) to the vertices of the four $D_4$ quivers in Figure~\ref{fig:F_36_fig1}~(middle), and exchanges the two external vertices of the second and fourth quiver as shown. Note that this special rule is similar to the rule of exchanging the corresponding central chords when rotating. The induced bijection between rays of the fan and centrally symmetric pairs of chords in the geometric model is illustrated in Figure~\ref{fig:bijection_rays_chords}.   

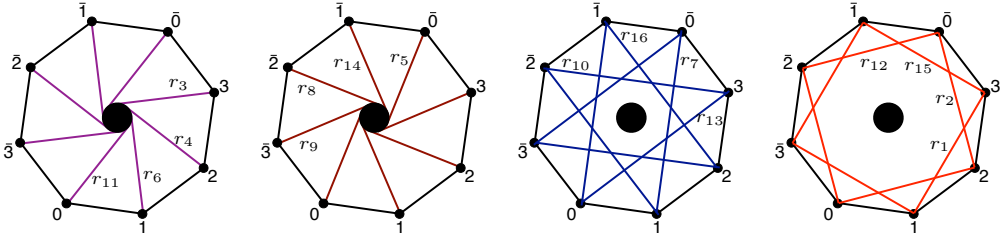
\begin{figure}[!htbp]
		\begin{overpic}[width=0.85\textwidth]{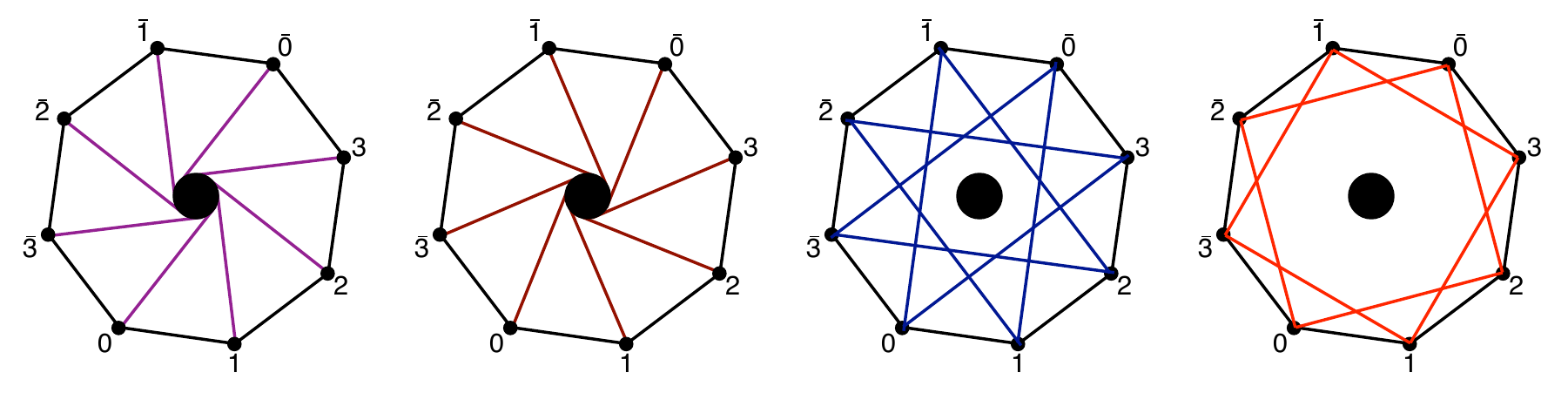}
			\put(10,6){\scriptsize $r_{11}$}
			\put(15,6.5){\scriptsize $r_6$}
			\put(18,10){\scriptsize $r_4$}
			\put(17.5,15.5){\scriptsize $r_3$}

			\put(39,18){\scriptsize $r_5$}
			\put(33.5,17.5){\scriptsize $r_{14}$}
			\put(30,15){\scriptsize $r_8$}
			\put(30.2,9.5){\scriptsize $r_9$}

			\put(68.7,12.1){\scriptsize $r_{13}$}			
			\put(67.2,17.7){\scriptsize $r_7$}
			\put(61.4,20.2){\scriptsize $r_{16}$}
			\put(55.7,17.5){\scriptsize $r_{10}$}

			\put(91.5,9.5){\scriptsize $r_1$}
			\put(92,14){\scriptsize $r_2$}
			\put(89,17.2){\scriptsize $r_{15}$}
			\put(84.7,17.5){\scriptsize $r_{12}$}

		\end{overpic}
	\caption{Bijection between rays of the fan $F_{36}$ and centrally symmetric pairs of chords in geometric model}
	\label{fig:bijection_rays_chords}
\end{figure}

\begin{theorem}
	Under the bijection $\Psi$, the cones of the fan $F_{3,6}$ correspond to clusters of type $D_4$, with the exception that the two cones over a bipyramid correspond to two clusters each, that are glued together on their common face:
\begin{figure}[!hbtp]
		\begin{overpic}[width=0.8\textwidth]{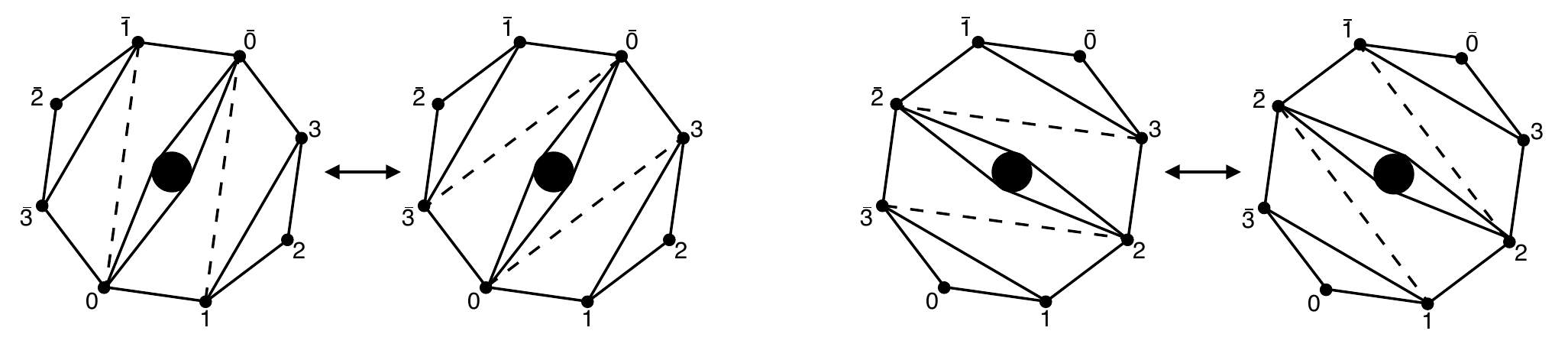}
			\put(12,-4){$\{r_1,r_5,r_7,r_{11},r_{13}\}$}
			\put(65,-4){$\{r_4,r_8,r_{10},r_{15},r_{16}\}$}
		\end{overpic}
	\label{fig:bipyramids}
\end{figure}
\end{theorem}

\begin{proof}
The proof of this result uses the description of cluster complexes in terms of the compatibility degrees of Fomin and Zelevinsky in~\cite[Section~3]{FominZelevinsky-YSystems}. The compatibility degree is a map 
\[
\begin{array}{ccc}
\Phi_{\geq -1} \times \Phi_{\geq -1}  & \longrightarrow  & \Z  \\
(\alpha,\beta)  & \longmapsto  & (\alpha\, || \, \beta)  
\end{array}
\]
characterized by the two properties:
\begin{eqnarray}
& (-\alpha_i\, || \, \beta) = b_i, & \text{for all } i \in [n] \text{ and } \beta = \sum b_i \alpha_i \in \Phi_{\ge-1}, \label{compatibility-relation1} \\
& (\alpha\, || \, \beta) = (\tau \alpha\, || \,  \tau \beta),  & \text{for all } \alpha,\beta \in \Phi_{\ge-1}, \label{compatibility-relation2}
\end{eqnarray}
where $\tau$ is the rotation operation on almost positive roots defined in Section~\ref{ssec:combi_type}. Two almost positive roots are said to be compatible if and only if their compatibility degree is zero. The cluster complex is the simplicial complex whose faces are sets of pairwise compatible roots. This complex is completely determined by its edges (1-dimensional simplices), and so it suffices to show that the edges in Figure~\ref{fig:F_36_fig1} (left) and Figure~\ref{fig:F_36_fig2} correspond exactly to the compatible pairs of almost positive roots under the map $\Psi$. This can be checked by inspection for the pairs involving a negative simple root, and by rotating the figures to obtain all other pairs. 
For instance, $-\alpha_1$ is compatible with $-\alpha_2,-\alpha_3,-\alpha_4,\alpha_2,\alpha_{23},\alpha_{24},\alpha_{234},\alpha_3,$ and~$\alpha_4$, while~$r_1=\Psi^{-1}(-\alpha_1)$ forms edges with the corresponding rays $r_{13},r_{5},r_{11},r_{7},r_{6},r_{14},r_{15},r_{3},$ and $r_{9}$.
The pairs of compatible roots that do not appear as edges in Figure~\ref{fig:F_36_fig1} (left) but do (in bold) in Figure~\ref{fig:F_36_fig2} are: 
\[
	\begin{array}{cccccc}
	(-\alpha_1, \alpha_{234})  & (-\alpha_3, \alpha_{124}) & (-\alpha_4, \alpha_{123}) &
	(\alpha_{12}, \alpha_{1})  & (\alpha_{23}, \alpha_{3})  & (\alpha_{24}, \alpha_{4})  \\
	(r_1,r_{15}) & (r_5,r_{8}) & (r_{11},r_{4}) &
	(r_2,r_{12}) & (r_6,r_{3}) & (r_{14},r_{9})  
\end{array}
\]
Taking the click complex of the compatibility relation finishes the proof.
\end{proof}

\section{Tropical Computations}
\label{sec:computations}
In this section, we compute the fan $F_{3,6}$ of Speyer and Williams \cite{speyer_tropical_2005} and analyze which tropical planes in $\mathbb{TP}^5$ are realized by $\Gr^+(3,6)$. 
We follow suit and compute $F_{3,6}$ in the same fashion as Speyer and Williams would in \cite{speyer_tropical_2005}. First we draw the web diagram $\Web_{3,6}$ and label its interior regions as shown in Figure \ref{fig:Web36}. This is the labeling used by Speyer and Williams in their computations.

\begin{figure}[!htbp]
\begin{center}
\begin{tikzpicture}[fleche/.style={<-,line width=2pt,>=stealth, cap=round}]

\def\dim{1.25}

\draw[fleche] (0,0) -- (\dim,0);
\draw[fleche] (\dim,0) -- (2*\dim,0);
\draw[fleche] (2*\dim,0) -- (2.5*\dim,0);

\draw[fleche] (0,-\dim) -- (\dim,-\dim);
\draw[fleche] (\dim,-\dim) -- (2*\dim,-\dim);
\draw[fleche] (2*\dim,-\dim) -- (2.5*\dim,-\dim);

\draw[fleche] (0,-2*\dim) -- (\dim,-2*\dim);
\draw[fleche] (\dim,-2*\dim) -- (2*\dim,-2*\dim);
\draw[fleche] (2*\dim,-2*\dim) -- (2.5*\dim,-2*\dim);

\draw[fleche] (0,-\dim) -- (0,0);
\draw[fleche] (\dim,-\dim) -- (\dim,0);
\draw[fleche] (2*\dim,-\dim) -- (2*\dim,0);

\draw[fleche] (0,-2*\dim) -- (0,-\dim);
\draw[fleche] (\dim,-2*\dim) -- (\dim,-\dim);
\draw[fleche] (2*\dim,-2*\dim) -- (2*\dim,-\dim);

\draw[fleche] (0,-2.5*\dim) -- (0,-2*\dim);
\draw[fleche] (\dim,-2.5*\dim) -- (\dim,-2*\dim);
\draw[fleche] (2*\dim,-2.5*\dim) -- (2*\dim,-2*\dim);

\node[label=right:{$1$}] at (2.5*\dim,0) {};
\node[label=right:{$2$}] at (2.5*\dim,-\dim) {};
\node[label=right:{$3$}] at (2.5*\dim,-2*\dim) {};

\node[label=below:{$6$}] at (0,-2.5*\dim) {};
\node[label=below:{$5$}] at (\dim,-2.5*\dim) {};
\node[label=below:{$4$}] at (2*\dim,-2.5*\dim) {};

\node at (0.5*\dim,-0.5*\dim) {$x_4$};
\node at (1.5*\dim,-0.5*\dim) {$x_2$};
\node at (0.5*\dim,-1.5*\dim) {$x_3$};
\node at (1.5*\dim,-1.5*\dim) {$x_1$};

\end{tikzpicture}
\caption{The labeling of the web diagram $\Web_{3,6}$}
\label{fig:Web36}
\end{center}
\end{figure}
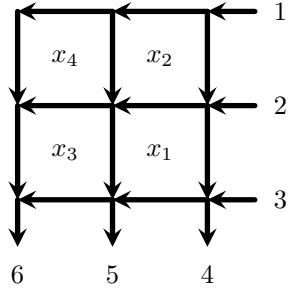

The fan $F_{3,6}$ is the complete fan in $\mathbb{R}^4$ whose maximal cones are the domains of linearity of the piecewise linear map
\[
	\Trop \Phi_2: \mathbb{R}^4\rightarrow \Gr^+(3,6)/(\Trop \phi)(\mathbb{R}^6),
\]
where $\Trop \phi$ is the map sending $(a_1,...,a_6)$ to the ${6 \choose 3}$-vector whose $(i_1,i_2,i_3)$-coordinate is $a_{i_1}+ a_{i_2} + a_{i_3}$; its image is the common lineality space of all cones in $\Gr^+(3,6)$. The map $\Trop \Phi_2$ is defined by the tropicalization of the maximal minors of the $3\times 6$ matrix $A_{3,6}$, whose entries are given by $$a_{ij}=(-1)^{i+1}\sum_p \mbox{Prod}_p,$$ where we are summing over all paths $p$ from $i$ to $j$ in $\Web_{3,6}$, and $\mbox{Prod}_p$ is the product of all the variables $x_i$ appearing below a given path $p$. Specifically, the matrix we get is

\[
A_{3,6}=
\left[
\begin{array}{rrrrrr}
1 & 0 & 0 & 1 & x_1x_2+x_1+1 & x_1x_2x_3x_4+x_1x_2x_3+x_1x_3+x_1x_2+x_1+1 \\
0 & -1 & 0 & -1 & -(x_1+1) & -(x_1x_3+x_1+1) \\
0 & 0 & 1 & 1 & 1 & 1
\end{array}
\right].
\]

The tropicalization of the maximal minors $P_{ijk}$ of $A_{3,6}$ are the following:

\[
\begin{array}{ll}
	P_{123}=0 & P_{234}=0 \\
	P_{124}=0 & P_{235}=\min\{0,x_1,x_1+x_2\} \\
	P_{125}=0 & P_{236}=\min\{0,x_1,x_1+x_2,x_1+x_3,x_1+x_2+x_3,x_1+x_2+x_3+x_4\} \\
	P_{126}=0 & P_{245}=\min\{x_1,x_1+x_2\} \\
	P_{134}=0 & P_{246}=\min\{x_1,x_1+x_2,x_1+x_3,x_1+x_2+x_3,x_1+x_2+x_3+x_4\} \\
	P_{135}=\min\{0,x_1\} & P_{256}=\min\{x_1+x_3,x_1+x_2+x_3,x_1+x_2+x_3+x_4\} \\
	P_{136}=\min\{0,x_1,x_1+x_3\} & P_{345}=x_1+x_2 \\
	P_{145}=x_1 & P_{346}=\min\{x_1+x_2,x_1+x_2+x_3,x_1+x_2+x_3+x_4\} \\
	P_{146}=\min\{x_1,x_1+x_3\} & P_{356}=\min\{x_1+x_2+x_3,x_1+x_2+x_3+x_4,2x_1+x_2+x_3+x_4\} \\
	P_{156}=x_1+x_3 & P_{456}=2x_1+x_2+x_3+x_4
\end{array}
\]

We then have that each $P_{ijk}$ gives rise to a fan $F(P_{ijk})$, and the simultaneous refinement of all of these fans is $F_{3,6}$. We compute this refinement using \texttt{Sage} \cite{sage} and get the rays:

\begin{center}
\[
	\begin{array}{l@{\hspace{1cm}}l@{\hspace{1cm}}l@{\hspace{1cm}}l}
	r_1= (0,0,1,0) & r_5= (-1,0,0,0) & r_9= (0,0,1,-1) & r_{13}= (0,1,1,-1) \\
	r_2= (0,0,-1,0) & r_6= (0,0,0,1) &  r_{10}= (1,0,0,-1) & r_{14}= (0,-1,0,0) \\
	r_3= (1,0,0,0) & r_7= (-1,0,0,1) & r_{11}= (0,1,0,0) & r_{15}= (1,-1,0,0) \\
		r_4= (1,0,-1,0) & r_8= (0,0,0,-1) & r_{12}= (0,1,0,-1) & r_{16}= (1,-1,-1,0)
	\end{array}
\]
\end{center}

Now that we have $F_{3,6}$, we would like to see which combinatorial types of generic planes in $\mathbb{TP}^5$ are realized by $\Gr^+(3,6)$. Speyer and Sturmfels \cite{speyer_tropical_2004} are the first to describe $\Gr(3,6)$ as the parameter space for tropical planes in $\mathbb{TP}^5$ and a recipe for computing which planes in $\mathbb{TP}^5$ realized where in $\Gr(3,6)$ is given by Herrmann, Jensen, Joswig, and Sturmfels in \cite{herrmann_how_2009}. We follow this recipe to compute which planes in $\mathbb{TP}^5$ are realized by $\Gr^+(3,6)$; for each maximal cone $C$ of $F_{3,6}$, the recipe goes as follows:

\begin{enumerate}
\item Choose an interior point $\lambda$ of $C$.
\item Compute its image $\Trop \Phi_2 (\lambda)$.
\item By Proposition \ref{proposition 3.1}, we know $\Trop \Phi_2 (\lambda)$ induces a matroid subdivision of the hypersimplex $\triangle(3,6)$; Compute this subdivision using \texttt{Polymake} \cite{polymake}.
\item Compare the computed matroid subdivision with the matroid subdivisions given in \cite{herrmann_how_2009} used to classify combinatorial types of generic tropical planes in $\mathbb{TP}^5$.
\end{enumerate}

Step (4) of the recipe was done by computing the face lattices of each matroid polytope for the matroids and the dimension of the intersections in the computed subdivision. Comparing them to the face lattices of the matroid polytopes of the matroids in the subdivisions given in \cite{herrmann_how_2009} classifies the combinatorial types of generic tropical planes in $\mathbb{TP}^5$. The graphical representation in Figure 1 of \cite{herrmann_how_2009} shows the neighboring properties using edges and 2-cells. The difference between EEFFa and EEFFb is as follows: in Type EEFFa, there are two matroid polytopes that do not intersect, whereas in Type EEFFb, there are three 2-dimensional intersections between the matroid polytopes. These computations were all made using \texttt{Sage} \cite{sage}. As the combinatorial type of plane does not change within a maximal cone \cite{herrmann_how_2009}, by following the recipe above for each maximal cone of $F_{3,6}$, we get all planes realized by $\Gr^+(3,6)$. 

\begin{theorem}\label{thm:computation}
Exactly six of the seven combinatorial types of tropical planes in $\mathbb{TP}^5$ are realized by $\Gr^+(3,6)$. As named by Sturmfels and Speyer \cite{speyer_tropical_2004}, the realizable combinatorial types are EEEG, EEFFa, EEFFb, EEFG, EFFG, and FFFGG. 
\end{theorem}

The partition into the combinatorial types is shown in Table~\ref{tab:cone_partition}.

\begin{table}[!htbp]
\begin{tabular}{ll}
Type EEEG: & $\{r_3,r_9,r_{10},r_{12}\},\{r_2,r_6,r_{14},r_{16}\},\{r_3,r_9,r_{12},r_{13}\},\{r_2,r_6,r_{7},r_{14}\}$ \T\B\\\hline
Type EEFFa: & $\{r_3,r_4,r_{6},r_{15}\},\{r_1,r_3,r_{6},r_{11}\},\{r_2,r_5,r_{8},r_{12}\},\{r_2,r_5,r_{11},r_{12}\},$\T\B\\
& $\{r_1,r_3,r_{6},r_{15}\},\{r_1,r_5,r_{9},r_{14}\},\{r_2,r_4,r_{8},r_{12}\},\{r_3,r_4,r_{6},r_{11}\},$ \T\B\\
& $\{r_5,r_8,r_{9},r_{14}\},\{r_8,r_9,r_{14},r_{15}\},\{r_2,r_4,r_{11},r_{12}\},\{r_1,r_9,r_{14},r_{15}\}$ \T\B\\\hline
Type EEFFb: & $\{r_2,r_5,r_{8},r_{14}\},\{r_1,r_3,r_{9},r_{15}\},\{r_2,r_4,r_{6},r_{11}\},\{r_5,r_8,r_{9},r_{12}\},$ \T\B\\
& $\{r_1,r_6,r_{14},r_{15}\},\{r_3,r_4,r_{11},r_{12}\}$ \T\B\\\hline
Type EEFG: & $\{r_5,r_9,r_{12},r_{13}\},\{r_3,r_9,r_{10},r_{15}\},\{r_3,r_4,r_{10},r_{12}\},\{r_3,r_{11},r_{12},r_{13}\},$ \T\B\\
& $\{r_1,r_{3},r_{9},r_{13}\},\{r_6,r_{14},r_{15},r_{16}\},\{r_1,r_{6},r_{7},r_{14}\},\{r_2,r_{8},r_{14},r_{16}\},$ \T\B\\
& $\{r_2,r_{5},r_{7},r_{14}\},\{r_8,r_{9},r_{10},r_{12}\},\{r_2,r_{4},r_{6},r_{16}\},\{r_2,r_{6},r_{7},r_{11}\}$ \T\B\\\hline
Type EFFG: & $\{r_8,r_{9},r_{10},r_{15}\},\{r_1,r_{5},r_{9},r_{13}\},\{r_1,r_{5},r_{7},r_{14}\},\{r_2,r_{4},r_{8},r_{16}\},$ \T\B\\
& $\{r_1,r_{3},r_{11},r_{13}\},\{r_2,r_{5},r_{7},r_{11}\},\{r_8,r_{14},r_{15},r_{16}\},\{r_3,r_{4},r_{10},r_{15}\},$ \T\B\\
& $\{r_4,r_{6},r_{15},r_{16}\},\{r_5,r_{11},r_{12},r_{13}\},\{r_1,r_{6},r_{7},r_{11}\},\{r_4,r_{8},r_{10},r_{12}\}$ \T\B\\\hline
Type FFFGG: & $\{r_4,r_{8},\underline{r_{10}},r_{15},\underline{r_{16}}\},\{r_1,r_{5},\underline{r_{7}},r_{11},\underline{r_{13}}\}$\T\B\\[1.5em]
\end{tabular}
\caption{The partition of the cone of the positive tropical Grassmannian into the corresponding combinatorial type of plane. In type FFFGG, the underlined rays represent the apexes of the splitted bipyramid to get the cluster complex.}
\label{tab:cone_partition}
\end{table}

\section{Comparing Tropical Planes and Pseudotriangulations}
\label{sec:comparing}

Noting that $\Dr(3,6)$ and $\Gr(3,6)$ are equal as sets, Speyer and Sturmfels describe $\Gr(3,6)$ as the parameter space for tropical planes in $\mathbb{TP}^5$.
Using Theorem~\ref{thm:computation}, we can deduce how the equivalence of tropical planes compares with equivalence of pseudotriangulations.

\begin{theorem}
The combinatorial types of tropical planes in $\mathbb{TP}^5$ and the combinatorial types of pseudotriangulations of~$\configD_4$ intersect transversally as illustrated in Table~\ref{tab:type}.
\end{theorem}

\begin{table}[!htbp]
\begin{tabular}{rc||c}
	& Type EEFG & Type EFFG \T\B\\
	Type T1: & \resizebox{0.40\hsize}{!}{\includegraphics{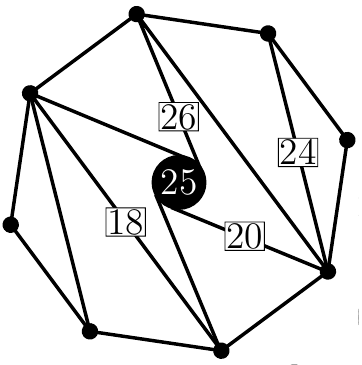} \includegraphics{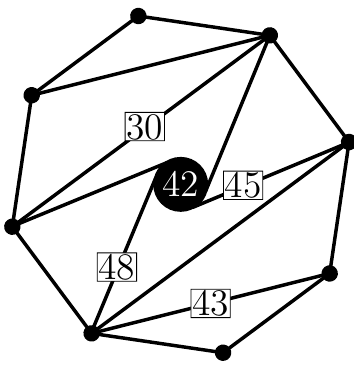} \includegraphics{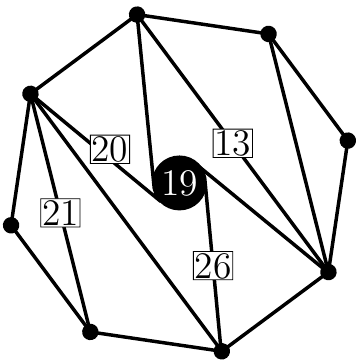} \includegraphics{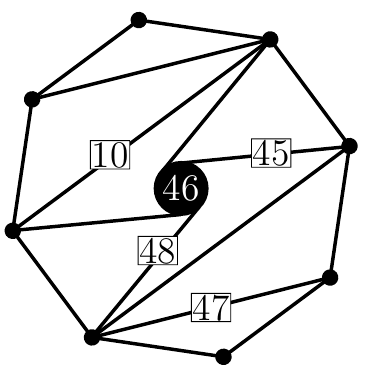}} & \resizebox{0.40\hsize}{!}{\includegraphics{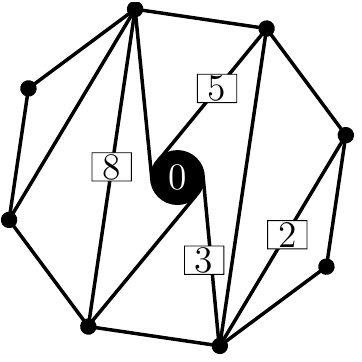} \includegraphics{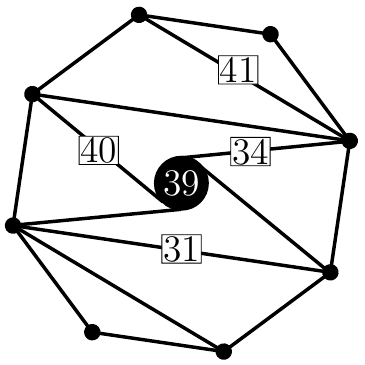} \includegraphics{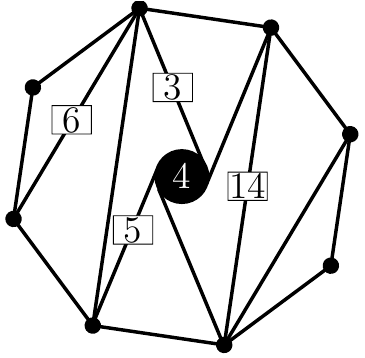} \includegraphics{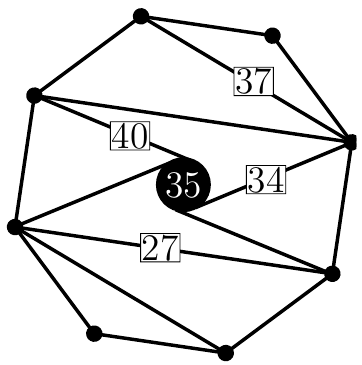}} \T\B\\
	& \resizebox{0.40\hsize}{!}{\includegraphics{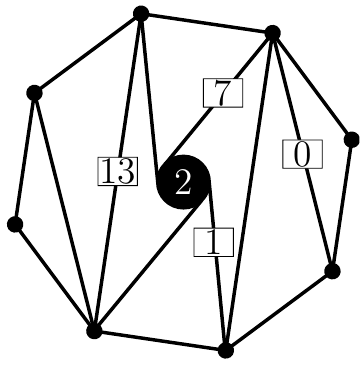} \includegraphics{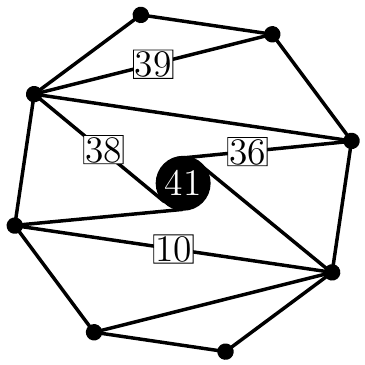} \includegraphics{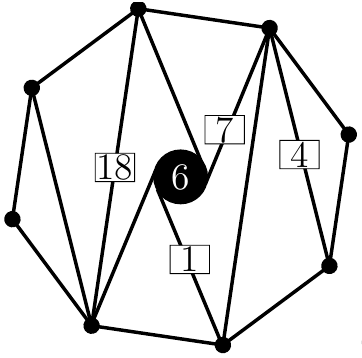} \includegraphics{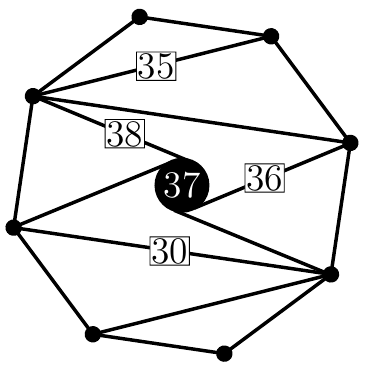}} & \resizebox{0.40\hsize}{!}{\includegraphics{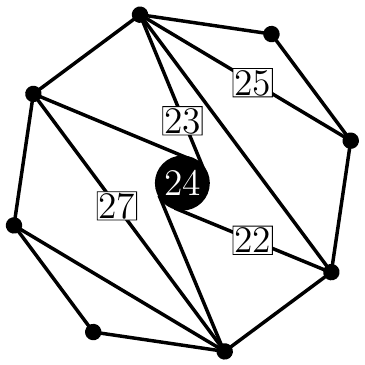} \includegraphics{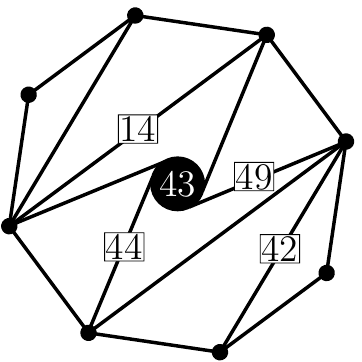} \includegraphics{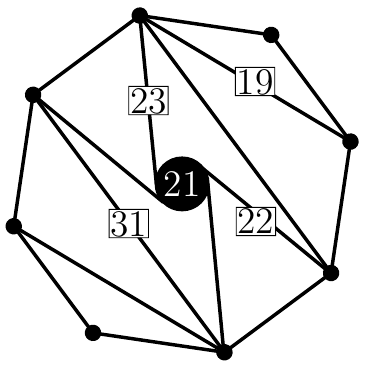} \includegraphics{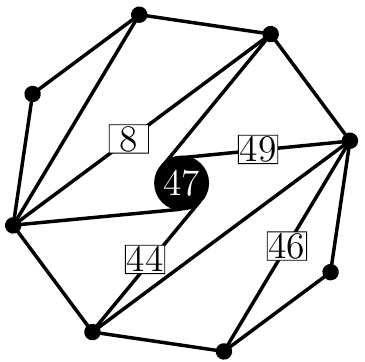}} \T\B\\\hline
	Type T2: & \resizebox{0.40\hsize}{!}{\includegraphics{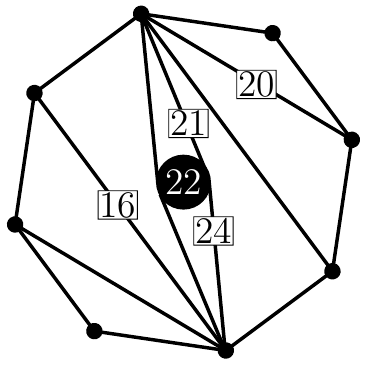} \includegraphics{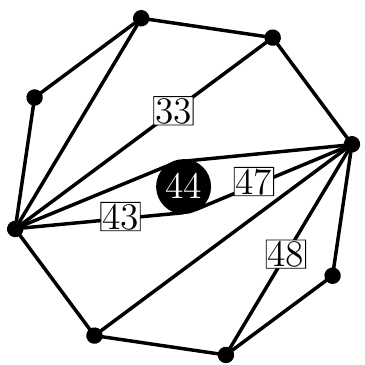} \includegraphics{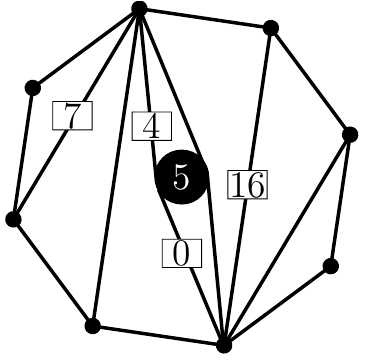} \includegraphics{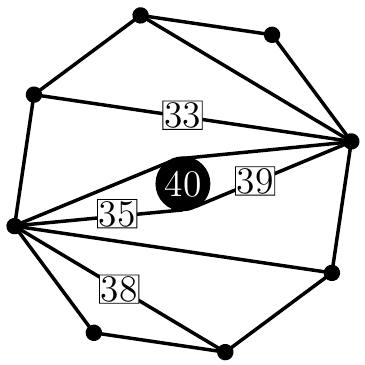}} & \resizebox{0.40\hsize}{!}{\includegraphics{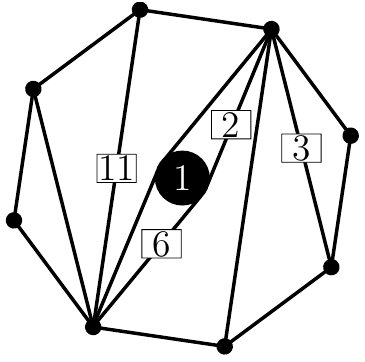} \includegraphics{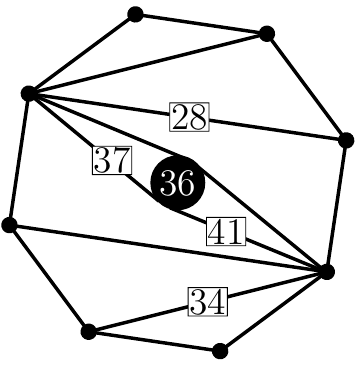} \includegraphics{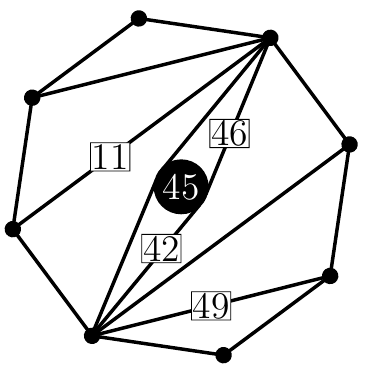} \includegraphics{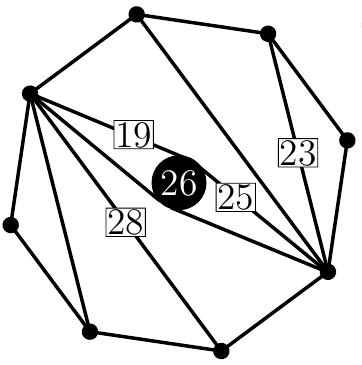}} \T\B\\\hline\hline
	& Type EEEG & Type FFFGG \T\B\\
	Type T3: & \resizebox{0.40\hsize}{!}{\includegraphics{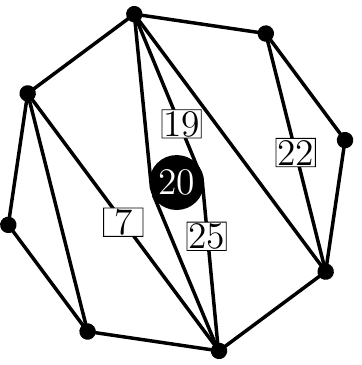} \includegraphics{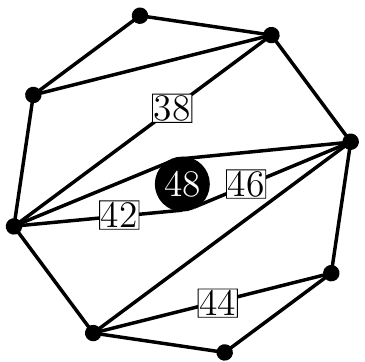} \includegraphics{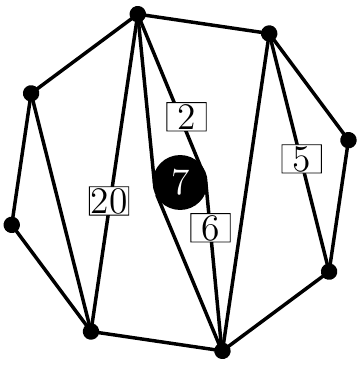} \includegraphics{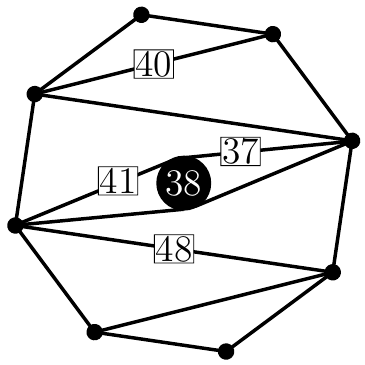}} & \resizebox{0.40\hsize}{!}{\includegraphics{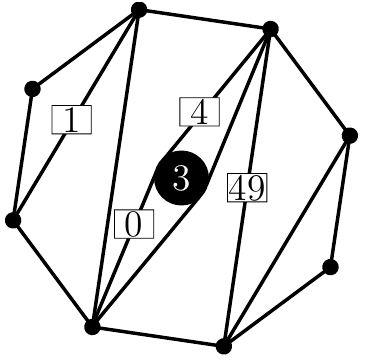} \includegraphics{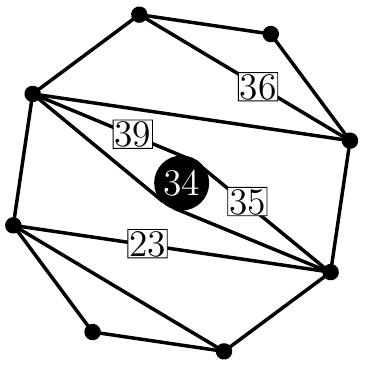} \includegraphics{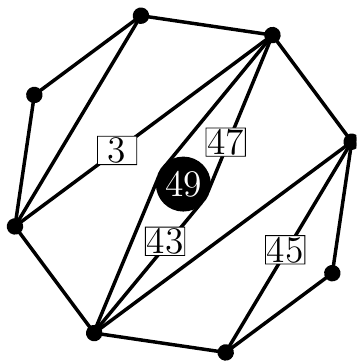} \includegraphics{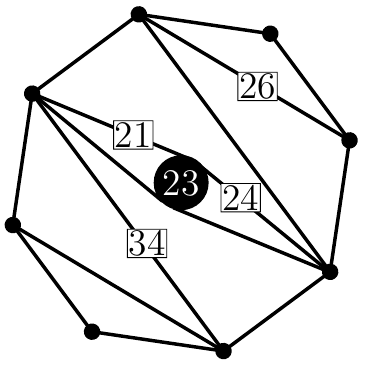}} \T\B\\\hline\hline
	& Type EEFFa & Type EEFFb \T\B\\
	Type T4: & \resizebox{0.20\hsize}{!}{\includegraphics{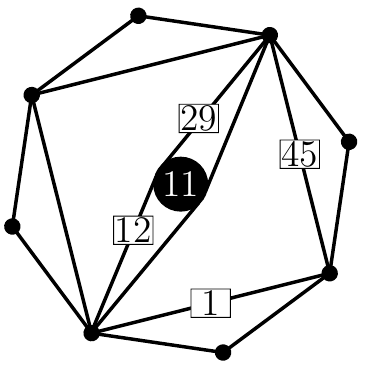} \includegraphics{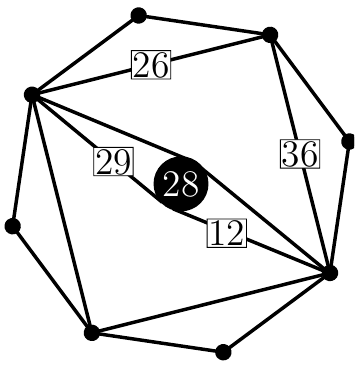}} & \resizebox{0.20\hsize}{!}{\includegraphics{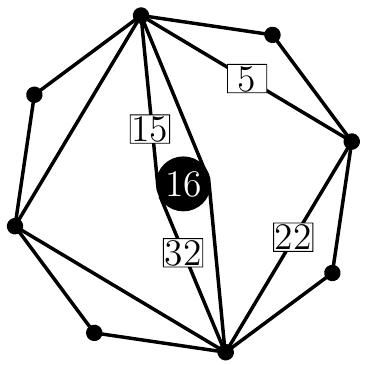} \includegraphics{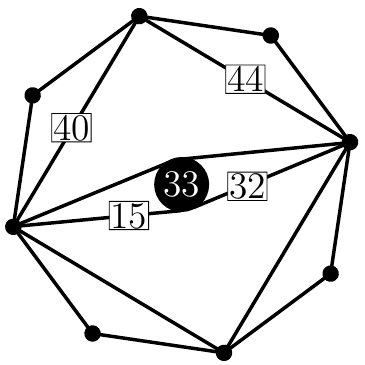}} \T\B\\\hline
	Type T5: & \resizebox{0.40\hsize}{!}{\includegraphics{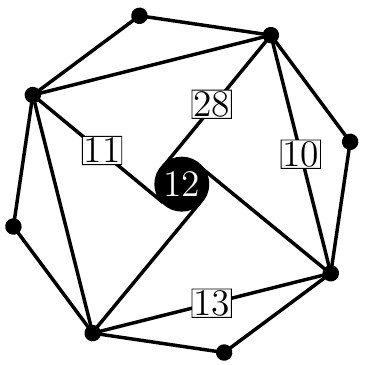} \includegraphics{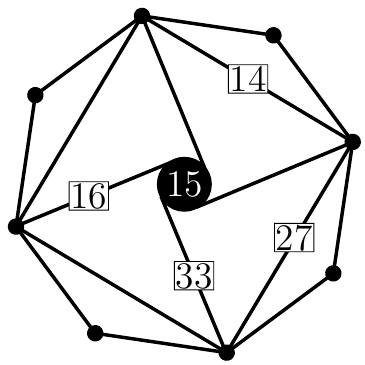} \includegraphics{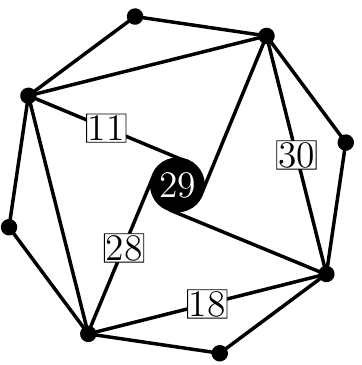} \includegraphics{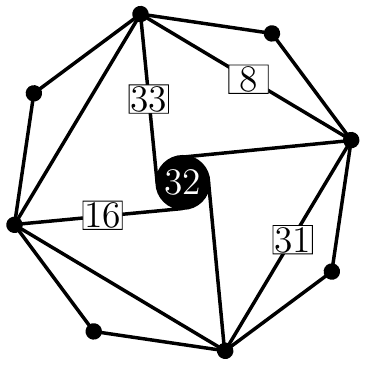}} & \T\B\\\hline
	Type T6: & \resizebox{0.40\hsize}{!}{\includegraphics{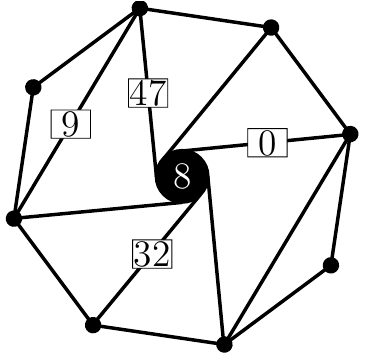} \includegraphics{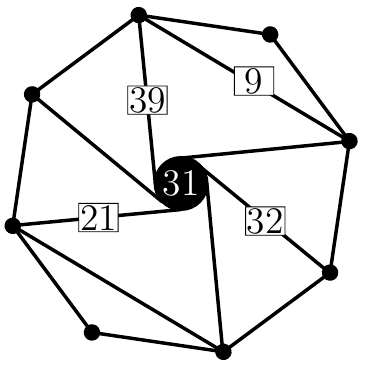} \includegraphics{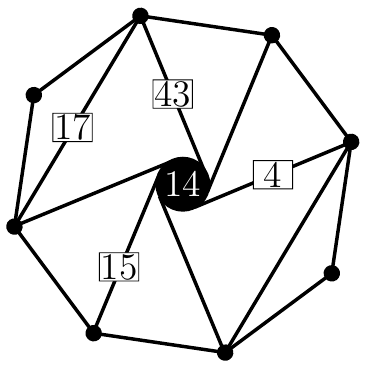} \includegraphics{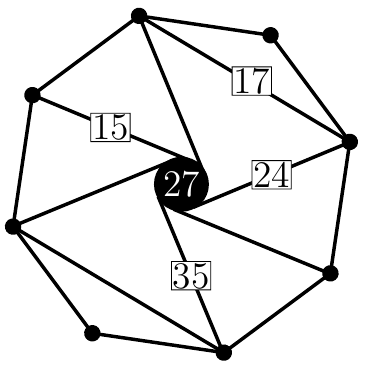}} & \resizebox{0.40\hsize}{!}{\includegraphics{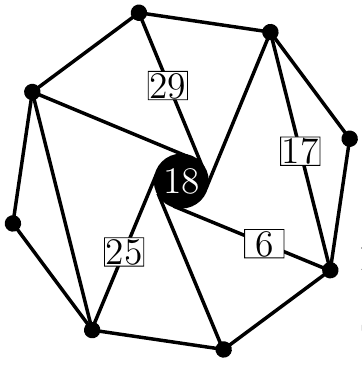} \includegraphics{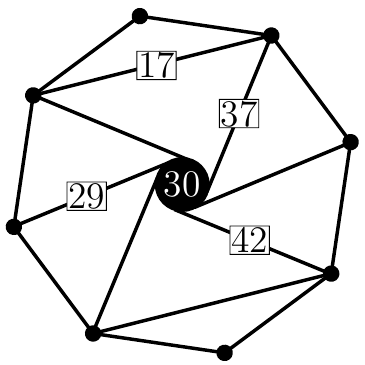} \includegraphics{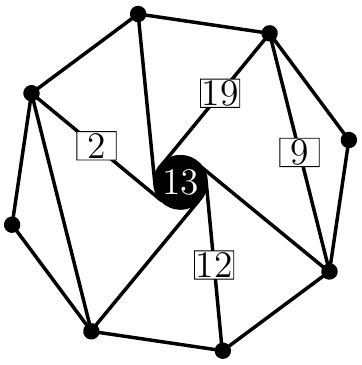} \includegraphics{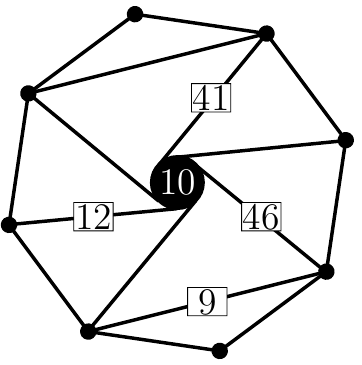}} \T\B\\\hline
	Type T7: & \resizebox{0.20\hsize}{!}{\includegraphics{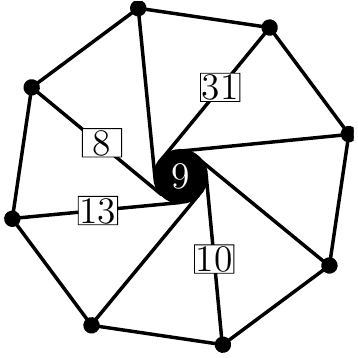} \includegraphics{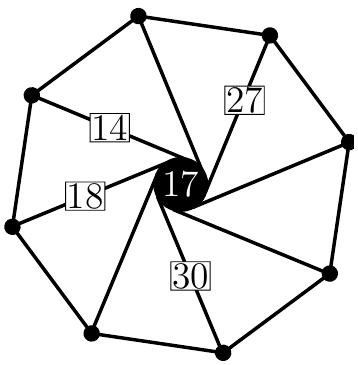}} &  \T\B\\
\end{tabular}
\caption{The combinatorial types of pseudotriangulations splitted into the combinatorial types of tropical planes.}
\label{tab:type}
\end{table}

Interestingly, although the equivalence relations are transversal, they intersect in a way that respects the reflection and swapping equivalence of the pseudotriangulations.
Using the table, we prove the next theorem giving a sufficient condition for two positive tropical planes to be combinatorially equivalent. 

\begin{theorem}
If two pseudotriangulations of~$\configD_4$ are related by a sequence of reflections of the octagon preserving the parity of the vertices (when labeled cyclically from 1 up to 8), and possibly followed by a global exchange of central chords, then their corresponding tropical planes in $\mathbb{TP}^5$ are combinatorially equivalent.
\end{theorem}

It turns out that this condition is necessary for types $EEEG$ and $FFFGG$. The four other positive types are obtained by taking unions of the classes generated by this finer equivalence on pseudotriangulations.

\begin{remark}
It is possible to translate this finer equivalence of pseudotriangulations in the language of subword complexes, see~\cite{ceballos_subword_2014}.
It would also be interesting to know if this sufficient condition extends to tropical planes in $\mathbb{TP}^6$ when looking at the cluster complex of type $E_6$ as a subword complex.
\end{remark}

\bibliographystyle{alpha}
\bibliography{article_arxiv}

\end{document}